\documentclass[11pt,leqno]{article}
\usepackage{amsthm,amsfonts,amssymb,epsfig,graphics,amsmath}
\usepackage[latin1]{inputenc}\relax

\newcommand{\E}{{\calE}}
\newcommand{\trace}{\text{\rm trace }}

\newcommand{\Trace}{{\text{\rm Trace}}}

\newcommand{\calE}{{\cal{E}}}

\newcommand{\CalA}{{\cal{A}}}

\newcommand{\CalT}{{\cal{T}}}

\newcommand{\btheta}{{\tilde \theta}}

\newcommand{\Tmu}{{\tilde \mu}}

\newcommand{\R}{\Re}
\newcommand{\const}{\text{\rm constant}}

\newcommand{\RR}{{\mathbb R}}
\newcommand{\WW}{{\mathbb W}}

\newcommand{\ZZ}{{\mathbb Z}}

\newcommand{\CC}{{\mathbb C}}
\newcommand{\mA}{{\mathbb A}}

\newcommand{\FF}{{\mathbb F}}
\newcommand{\GG}{{\mathbb G}}

\newcommand\cA{{\cal  A}}

\newcommand\cD{{\cal  D}}

\newcommand\cW{{\cal  W}}

\newcommand\cR{{\cal  R}}

\newcommand\cL{{\cal  L}}
\newcommand\cN{{\cal  N}}
\newcommand\cE{{\cal  E}}
\newcommand\cF{{\cal  F}}
\newcommand\cP{{\cal  P}}

\newcommand\cM{{\mathcal M}}
\newcommand\cT{{\mathcal T}}
\newcommand\cS{{\mathcal S}}
\newcommand\cZ{{\cal  Z}}

\newcommand{\Range}{\mathop\mathrm{Range}\nolimits}

\def\eps{\varepsilon }

\def\const{\text{\rm constant}}

\newcommand\adots{\mathinner{\mkern2mu\raise1pt\hbox{.}
\mkern3mu\raise4pt\hbox{.}\mkern1mu\raise7pt\hbox{.}}}

\newcommand{\rank}{{\rm rank }}

\newtheorem{theo}{Theorem}[section]

\newtheorem{cor}[theo]{Corollary}
\newtheorem{lem}[theo]{Lemma}

\newtheorem{ass}[theo]{Assumption}
\newtheorem{assums}[theo]{Assumptions}

\newtheorem{rem}[theo]{Remark}

\newtheorem{exams}[theo]{Examples}

\numberwithin{equation}{section}
\newcommand\kernel{\hbox{\rm Ker}}

\newcommand\br{\begin{remark}}
\newcommand\er{\end{remark}}
\newcommand\bp{\begin{pmatrix}}
\newcommand\ep{\end{pmatrix}}
\newcommand\be{\begin{equation}}
\newcommand\ee{\end{equation}}
\newcommand\ba{\begin{equation}\begin{aligned}}
\newcommand\ea{\end{aligned}\end{equation}}

\newcommand{\bap}{\begin{app}}
\newcommand{\eap}{\end{app}}
\newcommand{\begs}{\begin{exams}}
\newcommand{\eegs}{\end{exams}}
\newcommand{\beg}{\begin{example}}
\newcommand{\eeg}{\end{exaplem}}
\newcommand{\bpr}{\begin{proposition}}
\newcommand{\epr}{\end{proposition}}
\newcommand{\bt}{\begin{theorem}}
\newcommand{\et}{\end{theorem}}
\newcommand{\bc}{\begin{corollary}}
\newcommand{\ec}{\end{corollary}}
\newcommand{\bl}{\begin{lemma}}
\newcommand{\el}{\end{lemma}}
\newcommand{\bd}{\begin{definition}}
\newcommand{\ed}{\end{definition}}
\newcommand{\brs}{\begin{remarks}}
\newcommand{\ers}{\end{remarks}}

\newtheorem{theorem}{Theorem}[section]
\newtheorem{proposition}[theorem]{Proposition}
\newtheorem{corollary}[theorem]{Corollary}
\newtheorem{lemma}[theorem]{Lemma}
\newtheorem{definition}[theorem]{Definition}

\newtheorem{example}[theorem]{Example}
\newtheorem{remark}[theorem]{Remark}

\begin{document}

\title{Numerical error analysis for Evans function computations:
a numerical gap lemma, centered-coordinate methods,
and the unreasonable effectiveness of continuous orthogonalization }

\author{\sc \small 
Kevin Zumbrun\thanks{Indiana University, Bloomington, IN 47405;
kzumbrun@indiana.edu:
Research of K.Z. was partially supported
under NSF grants number DMS-0300487,  
DMS-0505780, and DMS-0801745.
}}

\maketitle

\begin{abstract}
We perform error analyses explaining some previously mysterious phenomena
arising in numerical computation of the Evans function, in particular
(i) the advantage of centered coordinates for exterior product and related
methods, and
(ii) the unexpected stability of the (notoriously unstable)
continuous orthogonalization method of Drury in the context of
Evans function applications.
The analysis in both cases centers around a numerical version of the 
gap lemma of Gardner--Zumbrun and Kapitula--Sandstede,
giving uniform error estimates for 
apparently ill-posed projective boundary-value problems with asymptotically
constant coefficients, so long as the rate of convergence of coefficients 
is greater than the ``badness'' of the boundary projections 
as measured by negative spectral gap.
In the second case, we use also the simple but apparently 
previously unremarked observation that the Drury method is in fact (neutrally)
stable when used to approximate an unstable subspace, 
so that continuous orthogonalization 
and the centered exterior product method 
are roughly equally well-conditioned as methods 
for Evans function approximation.
The latter observation makes possible an extremely simple
nonlinear boundary-value method
for possible use in large-scale systems, extending ideas suggested by Sandstede.
We suggest also a related linear method based on the conjugation lemma of
M\'etivier--Zumbrun, 
an extension of the gap lemma mentioned above.
\end{abstract}

\tableofcontents

\section{Introduction}\label{intro}

Recently, numerical Evans function computations have received
a great deal of attention as a tool for the stability analysis
of standing and traveling wave patterns in one and several dimensions; 
see, e.g.,
\cite{AS,Br,BrZ,BDG,AlB,HSZ,HuZ1,HuZ2,LPSS,GLZ,HLZ,CHNZ,HLyZ1,HLyZ2}.
In the work of the author together with Brin, Humpherys, Sandstede, 
and others, there has emerged a small list of three computational
rules of thumb, without which Evans function computations become 
hopelessly inefficient, but with which they become in 
usual situations almost trivial.
Indeed, Humpherys has developed a general MATLAB-based package 
(STABLAB) based on these principles that gives excellent results on
essentially all problems up to now considered.

Two of the items on this list are self-evident, but the third, the
need to ``center'' coordinates, does not appear to be well-known
and, indeed, at first sight appears to contradict standard stability
principles.
The purpose of this paper is twofold: first, to share 
these practical rules of thumb
and, second, to give a mathematical justification for the 
better-than-expected observed results of their implementation, that is,
to put these ad hoc principles on a rational and quantitative basis.
In the process, we discover a numerical analog of the gap lemma of 
\cite{GZ,KS}, a sort of superconvergence principle;  a new stability property 
of the well-known continuous orthogonalization method of Drury; and,
building on ideas of Sandstede
\cite{S} and Humpherys--Zumbrun \cite{HuZ1},
 an extremely simple boundary-value method for possible use
in ultra-large scale systems.

\subsection{Computation of the Evans function}\label{evans}
Let $L$ be a linear differential operator with asymptotically 
constant coefficients along some preferred spatial direction $x$, 
and suppose that the eigenvalue equation
\begin{equation}\label{eval}
(L-\lambda)w=0
\end{equation}
may be expressed as a first-order ODE in an appropriate phase space:
\begin{equation}\label{firstorder}
\begin{aligned}
W_x&=A(x,\lambda)W, \qquad
\lim_{x\to \pm \infty } A(x,\lambda)=A_\pm (\lambda),
\end{aligned}
\end{equation}
with $A$ analytic in $\lambda$ as a function from $\CC$ to $C^1(\RR,\CC^{n\times n})$ and the dimension $k$ of the stable subspace $S_+$ of $A_+$ and dimension $n-k$ of the unstable subspace $U_-$ of $A_-$ summing to the dimension $n$ of the entire phase space.
Then, the Evans function is defined as
\begin{equation}\label{evansdef}
D(\lambda):=\det\begin{pmatrix} W_1^+ & \cdots & W_k^+&W_{k+1}^- & \cdots & W_n^-\end{pmatrix}_{|x=0},
\end{equation}
where
$W_1^+, \dots, W_k^+$ and $W_{k+1}^-, \dots, W_n^-$ 
are analytically-chosen (in $\lambda$) bases
of the manifolds of solutions decaying as $x\to +\infty$ and $-\infty$.
For details of this construction, 
see, e.g., \cite{AGJ,PW,KS,GZ,Z1,HuZ1,HSZ} and references therein.  

Analogous to the characteristic polynomial for a finite-dimensional 
operator, $D(\cdot)$ is analytic in $\lambda$ with zeroes corresponding 
in both location and multiplicity to the eigenvalues of the linear operator $L$
\cite{GJ1,GJ2}.
Taking the winding number around a contour
$\Gamma =\partial \Lambda\subset\{\Re \lambda \ge 0\}$,
where $\Lambda$ is a set outside which eigenvalues may be excluded
by other methods (e.g. energy estimates or asymptotic ODE theory),
counts the number of unstable eigenvalues
in $\Lambda$ of the linearized operator about the wave,
with zero winding number corresponding to stability.
See, e.g., \cite{Br,BrZ,BDG,HSZ,HuZ1,BHRZ,HLZ,CHNZ,HLyZ1,HLyZ2,BHZ}. 
Alternatively, one may use Mueller's
method or any number of root-finding methods for analytic functions
to locate individual roots; see, e.g., \cite{OZ,LS}.

Numerical approximation of the Evans function breaks into two steps: 
(i) the computation of analytic bases for stable (resp. unstable) subspaces 
of $A_+$ (resp. $A_-$) and 
(ii) the propagation of these bases by ODE \eqref{firstorder} on a 
sufficiently large interval $x\in [M,0]$ (resp. $x\in [-M, 0]$).  
In both steps, it is important to preserve the fundamental property of 
analyticity in $\lambda$, which is extremely useful in computing 
roots by winding number or other methods.
Both problems concern (different aspects of)
{\it numerical propagation of subspaces}, the first in $\lambda$
and the second in $x$, thus tying into large bodies of theory
in both numerical linear algebra \cite{ACR,DDF,DE1,DE2,DF} and
hydrodynamic stability theory \cite{Dr,Da,NR1,NR2,NR3,NR4,B}.

Problem (i) has been examined in \cite{HSZ,Z2,BHZ}; for completeness,
we gather the (existing but dispersed) conclusions here in Section \ref{init}.
Our main emphasis, however, is on problem (ii) and numerical stability
analysis, for which we obtain substantially new results.

\subsection{Three Bad Things: numerical pitfalls 
and their resolutions}\label{three}

We now focus on problem (ii).
Our three basic principles for efficient numerical integration
of \eqref{firstorder} are readily motivated by consideration
of the simpler constant-coefficient case
\begin{equation}\label{cc}
\begin{aligned}
W_x&=A W, \qquad
A\equiv \const.
\end{aligned}
\end{equation}

\subsubsection{Potential pitfalls}\label{pitfalls}
We note the following
three basic pitfalls, two obvious and one perhaps less so.
\medskip

{\bf 1. Wrong direction of integration.}
Consider the simplest case that the dimension of the stable subspace
of $A$ is one, so that we seek to resolve a single decaying eigenmode,
with all other modes exponentially growing with increasing $x$.
Evidently, the correct direction of integration is the backwards
direction, from $x=+M$ back to $x=0$, in which the desired mode is
exponentially growing, and errors in other modes exponentially decay.
Integrating in the forward direction would be numerically disastrous,
 with exponential error growth $e^{\eta M}$, 
where $\eta$ is the spectral gap between decaying
and growing modes (recall, $M$ is large):
the analog for Evans function computations of integrating a
backward heat equation.
In general, we must always integrate from infinity toward zero.

\bigskip

{\bf 2. Parasitic modes} (related to 1).
For general systems of equations, the dimension of the stable subspace
of $A$ typically involves two or more eigenmodes, with distinct decay
rates $\mu_1<\mu_2<0$.
Integrating in backward direction as prescribed in part 1 above, we
resolve the fastest decaying $\mu_1$ mode without difficulty.
However, in trying to resolve the slower decaying $\mu_2$ mode,
we experience the problem that errors in the direction of the $\mu_1$
mode grow exponentially relative to the desired $\mu_2$ mode, at
relative rate $e^{(\mu_2-\mu_1)M}$.
That is, parasitic faster-decaying modes will tend to take over slower-decaying 
modes, preventing their resolution.

\br\label{badpar}
Degradation of results from parasitic modes
is of the same rough order as that resulting from integrating in the
wrong spatial direction, differing only in the fact that the maximum
spectral gap between two decaying modes is typically one-half or 
less of the maximum gap between decay and growing modes.
Thus, it is crucial to address this issue.
\er
\bigskip

{\bf 3. Nonequilibrium state.} 
A more subtle problem is that integrating even a single scalar
equation $w'=-aw$, $a\ge 0$, over a long interval $[M,0]$, leads to
(sometimes quite large) accumulation of errors, and, more important,
a large number of mesh points/computations.
The single exception is the equilibrium case $a=0$, which for essentially
all numerical ODE schemes is resolved exactly.

This can be understood more quantitatively
by the following heuristic computation, assuming a 
perfectly adaptive scheme and no machine error.
The truncation error $\tau_j$
 for a $k$th order scheme at step $j$ is proportional to
the $(k+1)$th derivative of the solution times $\Delta x_j^k$, where 
$\Delta x_j$ is the size of the $j$th step, or 
$c_k a^{k+1}e^{-ax_j}\Delta x_j^k$.
Taking $\tau_j \sim TOL$ for some fixed tolerance $TOL$, we thus obtain
$c_k a^{k+1}e^{-ax_j}\Delta x_j^k \sim TOL$, or
$$
 \frac{\Delta x_j}{\Delta j} \sim c_k^{-1/k}  TOL^{1/k}\times a^{-1-1/k}e^{ax_j/k}.
$$
Inverting, and integrating $\frac{\Delta j}{\Delta x_j}\approx 
\frac{dj}{dx}$ from $0$ to $M$,\footnote{
Direction of integration is symmetric for a single mode.
}
we obtain an estimate
\be\label{Jpred}
J\sim c_k^{1/k}
TOL^{-1/k} k a^{1/k}
\int_0^M
(a/k) e^{-ax_j/k}
\sim c_k^{1/k} 
TOL^{-1/k} k a^{1/k}
\ee
as $M\to +\infty$ for the total number $J$ of mesh blocks,
which goes to zero as $a\to 0$ and to infinity as $a\to \infty$.

Though it is tempting to think of this as 
an example of numerical stiffness, 
that is not the case, since the problem involves but a single mode.
It seems rather to be a secondary, previously unremarked, phenomenon,
that in usual circumstances would be neglible.
For Evans function computations, however, we observe a difference
in computational efficiency of an order of magnitude or more between the cases
$a=1$ and $a=0$, for $TOL\sim 10^{-6}$.

%

\br
It would be iteresting to compare \eqref{Jpred} to results for the
standard RK45 scheme with which most Evans computations have been done
applied to the constant coefficient scalar problem $w'=-aw$,
in particular the $a^{1/4}$ rate.
For variable-coefficient systems, additional effects 
having to do with ``conjugation errors'' appear as well; 
see Section \ref{numgap}.
\er

\bigskip

\subsubsection{Solution one: the centered exterior product method}\label{solution1}

Problem 1 is easily avoided by integrating in the correct direction.
Problem 2 may be overcome by working in the exterior product space 
$W_1^+\wedge \cdots \wedge W_k^+ \in \CC^{\binom{n}{k}}$ 
(resp. $W_{k+1}^-\wedge \cdots \wedge W_n^- \in \CC^{\binom{n}{n-k}}$), 
for which the desired subspace appears as a single, maximally stable 
(resp. unstable) mode, the Evans determinant then being recovered 
through the isomorphism
\be\label{exteval}
\det\begin{pmatrix} W_1^+ & \cdots & W_k^+&W_{k+1}^- & \cdots & 
W_n^-\end{pmatrix}
\sim
(W_1^+\wedge \cdots \wedge W_k^+) \wedge (W_{k+1}^-\wedge \cdots \wedge W_n^-);\\
\ee
see \cite{AS,Br,BrZ,BDG,AlB} and ancestors \cite{GB,NR1,NR2,NR3,NR4}.
This reduces the problem to the case $k=1$.
Problem 3 can then be avoided by factoring out the expected asymptotic
decay rate $e^{\mu x}$ of the single decaying mode and solving the
``centered'' equation
\be\label{centered}
Z'=(A-\mu I)Z, \quad Z(+\infty)=r: \, A_+r=\mu r
\ee
for $Z:= e^{-\mu x}W$, which is now asymptotically an equilibrium as $x\to +\infty$.
With these preparations, one obtains excellent results \cite{BDG,HuZ1};
however, omitting any one of them leads to a loss of efficiency of
at least an order of magnitude in our experience \cite{HuZ2}.

\subsubsection{Solution two: the polar coordinate method} \label{solution2}

Unfortunately, the dimension $n\choose k$ of the phase space for the
exterior product grows exponentially with $n$, since $k$ is 
$\sim n/2$ in typical applications.
This limits its usefulness to $n\le 10$ or so, whereas the Evans
system arising in compressible MHD is size $n=15$, $k=7$ \cite{BHZ},
giving a phase space of size ${n\choose k}=6,435$: clearly impractical.
A more compact, but nonlinear, alternative is the polar coordinate
method of \cite{HuZ1}, in which the
exterior products of the columns of $W_\pm$ are represented in
``polar coordinates'' $(\Omega, \gamma)_\pm$, 
where the columns of $\Omega_+ \in \CC^{n\times k}$ 
and $\Omega_- \in \CC^{(n-k)\times k}$ are orthonormal bases 
of the subspaces spanned by the columns of 
$W_+:= \begin{pmatrix} W_1^+ & \cdots & W_k^+ \end{pmatrix}$ and 
$W_-:= \begin{pmatrix} W_{k+1}^- & \cdots & W_n^- \end{pmatrix}$, $W_j^\pm$ 
defined as in \eqref{evansdef}, i.e., $W_+=\Omega_+ \alpha_+$, 
$W_-=\Omega_- \alpha_-$, and $ \gamma_\pm:= \det \alpha_\pm$,
so that
\[
W_1^+\wedge \cdots \wedge W_k^+ \wedge= \gamma_+ 
(\Omega_+^1\wedge \cdots \wedge \Omega_+^k),
\]
where $\Omega_\pm^j$ denotes the $j$th column of $\Omega_\pm$,
and likewise
$
W_{k+1}^-\wedge \cdots \wedge W_n^-=\gamma_- 
(\Omega_-^1 \wedge \cdots \wedge \Omega_-^{n-k}).
$

This yields the block-triangular system
\begin{equation}\label{Omegaeq}
\begin{aligned}
\Omega'&= (I-\Omega\Omega^*) A \Omega, \\
(\log \tilde \gamma)'&= \trace (\Omega^*A\Omega)
- \trace (\Omega^*A\Omega)(\pm \infty),
\end{aligned}
\end{equation}
$\tilde \gamma :=\tilde \gamma e^{- \trace (\Omega^*A\Omega)(\pm \infty)x}$,
for which the ``angular'' $\Omega$-equation is exactly
the continuous orthogonalization method of Drury \cite{Dr,Da},
and the ``radial'' $\tilde \gamma$-equation, given $\Omega$, may be solved by 
simple quadrature.
Ignoring the numerically trivial radial equation, we see that problem
2 by fiat does not occur.
Likewise, for constant $A$, it is easily verified that invariant subspaces
$\Omega$ of $A$ are equilibria of the flow, so problem 3 does not occur.
Indeed, for $A$ constant, solutions of the $\tilde \gamma$-equation are 
constant, so that any first-order or higher numerical scheme 
resolves $\log \tilde \gamma$ exactly;
thus, the $\tilde \gamma$-equation may for simplicity
be solved together with the $\Omega$-equation, with no need for 
a final quadrature sweep.\footnote{See Section \ref{polarinit} for
numerical prescriptions $(\Omega,\tilde \gamma)_\pm$ of $(\Omega, \tilde \gamma)$ at $\pm \infty$.}
The Evans function is recovered, finally, through the relation
\be\label{polarevans}
D(\lambda)=
\det\begin{pmatrix} W_1^+ & \cdots & W_k^+&W_{k+1}^- & \cdots & 
W_n^-\end{pmatrix}|_{x=0}=
\tilde \gamma_+\tilde \gamma_-\det(\Omega^+,\Omega^-)|_{x=0}.
\ee


\subsection{Further questions and description of results}

The discussion of the previous subsection leaves open some
important questions.
First, can we justify this heuristic discussion with rigorous,
quantitative error bounds for the actual, variable coefficient
problems that occur in practice?
In particular, we note that centering equations as in Section 
\ref{solution1} goes counter to the intuition afforded by standard
two-point boundary-value theory on intervals $[0,M]$ as
$M\to \infty$ \cite{Be1},
which asserts convergence error of order $e^{-\eta M}$ where
$\eta$ is the minimum spectral gap of decaying (resp. growing)
modes from zero to the solution on $[0,+\infty)$.
Applied blindly to the centered equations, this would predict
{\it nonconvergence} rather than the good behavior observed 
in practice.

Second, there is a well-known problem of instability
of the continuous orthogonalization method with respect to
perturbations disturbing the assumed orthonormal structure of $\Omega$
\cite{Da,BrRe}.
In the language of \cite{BrRe}, the {\it Stiefel manifold}
$\cS:=\{\Omega:\, \Omega^*\Omega=I_k\}$
of orthonormal matrices is preserved by the flow of \eqref{Omegaeq},
but is typically neither attracting nor repelling. 
In view of Remark \ref{badpar}, this should lead to terrible
results for Evans function computations.  This issue was discussed
at length in \cite{HuZ1}, with numerous different
solutions discussed, from artificial stabilization to geometric integration.
Yet, surprisingly, the method that performed best was the original
Drury algorithm with no stabilization, implemented by a standard
RK45 scheme. 
This yielded results quite similar to those of the exterior product scheme,
which seems to contradict the conclusions of Remark \ref{badpar}.
\medskip

{\bf Result I.}
Our first main result is to establish a numerical version of the
gap lemma of \cite{GZ,KS}, which states that, ignoring machine error,
provided the coefficient
matrix $A(x,\lambda)$ is uniformly exponentially convergent as
$x\to +\infty$, with rate $|A-A_+|\le Ce^{-\theta x}$
for $x\ge 0$, and provided the gap between $\mu$ minimum real part of the
eigenvalues of $A_+$ is strictly greater than $-\theta$, 
then the solution of problem \eqref{centered} on $[0,M]$ 
initialized as $Z(M)=r$
converges as $M\to \infty$ to the solution on $[0,+\infty)$ at
rate $C(\tilde \theta, \theta)e^{-\tilde \theta M}$ 
for any $0<\tilde \theta<\theta$.
This resolves the first issue, explaining the observed convergence of
the centered exterior product method.

Completing the analogy to \cite{GZ},
we establish a corresponding result for general centered 
two-point boundary problems \eqref{centered}
with projective boundary conditions on $[0,M]$, 
under the assumption that the gap $\gamma$ 
between $\mu$ minimum real part of the
eigenvalues of $A_+$ {associated with the projective boundary condition
at $M$} is strictly greater than $-\theta$, 
obtaining convergence at 
the same
rate $C(\tilde \theta, \theta)e^{-\tilde \theta M}$ 
as $M\to +\infty$ for any $0<\tilde \theta<\theta$.
This could be viewed as a type of superconvergence, as the standard
theory \cite{Be1} predicts convergence at rate $e^{-\gamma x}$, with
a nonpositive spectral gap $\gamma \le 0$ corresponding to ill-conditioned
boundary conditions.
For detailed statements and proofs, see Section \ref{numgap}.
\medskip

{\bf Result II.}
Our second main result is to explain the apparent contradiction between
observed good results for the polar coordinate method \cite{HuZ1} and
the well-known instability of continuous orthogonalization \cite{Da,BrRe}.
The simple resolution is that, though continuous orthogonalization 
{\it is} in general unstable, it is in the present context stable!

Heuristically, this is quite simple to see.
Intuitively, it is is clear that the stable manifold of $A_+$
is asymptotically attracting in backward $x$ under the flow
of \eqref{Omegaeq} for $\Omega$ 
confined to the Stiefel manifold $\cS=\{\Omega:\, \Omega^*\Omega=I_k\}$,
and in fact this is well known (see Section \ref{tan}).
Thus, we need only verify that the Stiefel manifold, likewise, is
attracting in backward $x$.
Defining the Stiefel error $\E(\Omega):=\Omega^*\Omega-I$,
we obtain after a brief computation the error equation
\begin{equation}\label{druryerror}
{\cal{E}}'= -  {\E} (\Omega^* A\Omega) -  (\Omega^* A\Omega)^*\E
\end{equation}
of \cite{HuZ1}.
Linearizing about the exact solution $\bar \Omega\to \Omega_+$,
$\bar \cE \to 0$ and replacing coefficients by their asymptotic limits,
we obtain a linear equation
$\cE'=\cA_+\cE$, where $\cA\cE:=-\cE\tilde A_+-(\tilde A_+)^*\cE$
is a Sylvester operator, $\tilde A_+:=(\Omega_+^* A_+\Omega_+)$, 
with eigenvalues and eigenmatrices 
$a_j+a_k^*$, $r_j r_k^*$, where $a_j$ and $r_j$ are eigenvalues and
eigenvectors of $\tilde A_+$.
Noting that the eigenvalues of $\tilde A_+$ are exactly the eigenvalues
of $A_+$ restricted to its stable subspace, i.e., the stable eigenvalues
of $A_+$, we find that $\cA_+$ has positive real part eigenvalues,
and so $\cE$ decays in backward $x$.

That is, {\it the Stiefel manifold is attracting under the backward
flow of continuous orthogonalization (repelling under the forward
flow) if $\Omega$ is a stable subspace of $A$}, an observation that
previously seems to have gone unremarked.
This confirms that, as suggested by numerical results of \cite{HuZ1},
continuous orthogonalization is roughly equally well-conditioned as
the centered exterior product method.
For details and further discussion, see Section \ref{orthstab}.

\br\label{nogeo}
An important consequence is that geometric integrators
like those suggested in \cite{BrRe} for general Orr--Sommerfeld applications
are probably not worth the trouble for Evans function computations,
since the Drury method is stable and much simpler to code.
\er

\br
The generalized inverse method, or {\it Davey method} \cite{Da}
$\Omega'= (I-\Omega\Omega^\dagger) A \Omega$,
where $\Omega^\dagger:= (\Omega^* \Omega)^{-1}\Omega^*$ denotes the 
generalized inverse, exhibits neutral error growth $\E'= 0$,
so is often used as a stabilization of the basic continuous orthogonalization
method \eqref{Omegaeq}(i) of Drury.
In the context of Evans function computations, the behavior is slightly worse,
however \cite{HuZ1}.  This can now be understood from the fact that the Drury
method actively damps errors in the Evans function context.
The fact that the Drury method outperformed the Davey method (and all others)
was a mysterious aspect left unresolved in \cite{HuZ1}.
\er

\medskip

{\bf Result III.}
For equations arising in complicated physical 
systems or through transverse discretization of a multi-dimensional 
problem on a cylindrical domain \cite{LPSS}, 
the dimension $n$ can be very large.
For example, for the multidimensional systems considered in \cite{LPSS},
$n=8M\sim 48$ (see p. 1447, \cite{LPSS}) 
where $M\sim 6$ is the number of transverse Fourier modes being computed.
The development of numerical methods suitable for efficient
Evans function computations for large systems has been cited 
by Jones and others as one of the key problems facing the traveling-wave 
community in the next generation \cite{J}. 

For large dimensions, the accumulation of errors associated with shooting
methods appears potentially problematic, and so various other options
have been considered.
For example, one may always abandon the Evans function formulation and 
go back to direct discretization/Galerkin techniques, hoping to 
optimize perhaps by multi-pole type expansions on a problem-specific basis.  
However, this ignores the useful structure, and associated dimensionality 
reduction, encoded by existence of the Evans function.\footnote{
See however \cite{GLZ}, which points out a useful intermediate
structure based on Fredholm determinants.}

Alternatively, Sandstede \cite{S} has suggested to work within the Evans 
function formulation, but, in place of the high-dimensional shooting 
methods described above, to recast \eqref{firstorder} as a 
boundary-value problem with appropriate projective boundary conditions, 
which may be solved in the original space $\CC^n$ for individual 
modes by robust and highly-accurate boundary-value/continuation techniques.  

Problems with this scheme as conventionally implemented in uncentered
coordinates are two.
First, solutions exponentially decay as $x\to \pm\infty$, so that
the direct connection to data at $\infty$ of \eqref{centered} is lost;
as a consequence, up to now, it is not known how to recover analyticity 
of the Evans function by such a scheme.
Second, since the uncentered problem involves decaying
modes as well as growing modes, there must be provided boundary
conditions at $x=0$ as well as at $x=\pm M$; indeed, it is the
boundary conditions at $x=0$ that mainly determine the decaying modes we
seek.
Since behavior of \eqref{firstorder} is only known near its asymptotic
limits as $x\to \pm \infty$, there appears to be no analytic way to 
prescribe a priori well-conditioned projective boundary conditions at
$x=0$ (or, as mentioned already, to relate these to a desired asymptotic
behavior as $x\to \pm \infty$), and so apparently these must be adjusted 
``on the fly'' by trial and error, a process that requires error checks and 
additional complications in program structure.

As pointed out in \cite{HuZ1} (last sentence of introduction),
using the polar coordinate method-- or any centered scheme--
as the basis of a boundary-value scheme eliminates immediately the
first problem, of preserving analyticity, since in the centered format
date is explicitly described as $x\to \pm \infty$.
The second problem in general remains.
However, by the remarkable stability property recorded in result II,
the polar coordinate method involves only modes that are neutral or
growing as $x\to \pm \infty$, and not decaying; that is, it is 
both {\it centered} and {\it one-sided}.
The centered exterior product method though dimensionally unsuitable
shares this property as well; indeed, it is precisely the one-sided
property that makes these schemes suitable for shooting.
For such schemes, appropriate projective boundary conditions
by result I are 
full Dirichlet conditions as $x\to \pm \infty$, {\it with
no boundary conditions at $x=0$}, and thus the second, essentially
logistic, problem does not either arise.

{We therefore propose the polar coordinate method with Dirichlet 
conditions at $x=\pm M$ as a promising candidate for boundary-value-based
Evans function computations}, noting in particular that it is essentially
trivial to program given an existing shooting code.
The only disadvantage that we immediately see is the nonlinearity of
the scheme.
We propose at the same time
an alternative linear, centered but two-sided, scheme based
on the conjugation lemma of \cite{MeZ,Z1}, an extension of
the gap lemma that is our main tool in the analysis.
\medskip

\subsection{Discussion and open problems}\label{conclusions}

We gather in this paper a complete prescription
together with rigorous error bounds
for efficient numerical Evans function computations
by the shooting methods of \cite{HuZ1,HLZ,HLyZ1,BHZ}, etc., 
of systems up to the intermediate size $n\sim 20$ or so
encountered in one- and multi-dimensional problems of continuum
mechanics, and propose some promising boundary-value methods for
further exploration in computations for ultra-large systems
of size $n\sim 50$ and up.
In the process, we rehabilitate the
continuous orthogonalization method of Drury, explaining its
unexpectedly good performance in the context of Evans function
computations.
Finally, and most important, we point out 
the importance of centered coordinates, in contrast to
standard numerical intuition and practice in the study
of boundary-value problems on unbounded domains \cite{Be1},
establishing the related stability/superconvergence principle
embodied by our numerical gap lemma.

%
The latter result
seems to suggest larger implications in the construction of numerical 
boundary-value schemes.
%
At the same time, it serves to clarify some up-to-now rather
confusing existing results.
For example, in the seminal work \cite{Er}, Erpenbeck performed a numerical
Evans function analysis of stability of ZND detonation waves
by a method obeying principle 1 and 2 of Section \ref{pitfalls}. 
Much later, Lee and Steward \cite{LS} introduced what is now the effective
standard method in detonation literature, obeying
principle 2 but violating principle 1, reporting an apparently 
counter-intuitive improvement in speed of computation.  
The explanation of this paradox is that Erpenbeck carried out
his computations in uncentered coordinates, whereas
Lee and Steward, by mapping to a bounded interval
and applying a singular integral solver 
effectively centered their equations,
factoring out the principal dynamics.
Thus, there is a cancellation of errors involved, that
cannot be seen without reference to principle 3.
A centered version of Erpenbeck's original method 
appears to outperform both schemes by an order
of magnitude; see \cite{HuZ2} for further discussion/simplifications.

The main mathematical interest of our
numerical convergence results is their application to two-sided
boundary-value schemes posed on the entire interval $[0,M]$.
Indeed, for shooting methods, a routine translation into the discrete
setting of the continuous gap lemma yields the result, whereas
the general case involves a more subtle argument based on approximate
conjugation to constant-coefficients; see Remark \ref{pfrmk}.
The determination of realistic mesh requirements 
for centered boundary-value schemes, and the question 
of whether or not centered boundary-value schemes yield 
in practice
the same good performance observed for centered shooting schemes,
remain important open problems.

\section{Preliminaries: the gap and conjugation lemmas}\label{prelim}

We begin by recalling the standard gap and conjugation lemmas
of \cite{GZ,KS} and \cite{MeZ}.
Consider a general family of first-order ODE 
\begin{equation}
\label{gfirstorder}
\WW'-{\mathbb A}(x, \Lambda)\WW=0
\end{equation}
indexed by a parameter $\Lambda \in \Omega \subset \CC^m$, where
$W\in \CC^N$, $x\in \RR$ and ``$'$'' denotes $d/dx$.
Assume

{(h0) } Coefficient ${\mathbb A}(\cdot,\Lambda)$, considered
as a function from $\Omega$ into $C^0(x)$
is analytic in $\Lambda$.
Moreover, ${\mathbb A}(\cdot, \Lambda)$ approaches
exponentially to limits $\mA_\pm$ as $x\to \pm \infty$, 
with uniform exponential decay estimates
\begin{equation}
\label{expdecay2}
|(\partial/\partial x)^k(\mA- \mA_\pm)| 
\le Ce^{-\theta|x|}, \, 
\quad
\text{\rm for } x\gtrless 0, \, 0\le k\le K,
\end{equation}
$C$, $\theta>0$, 
on compact subsets of $\Omega $.
\medbreak

\begin{lem}[{The gap lemma [GZ, ZH]}]
\label{gaplemma}
Assuming (h0),
if $V^-(\Lambda)$ is an 
eigenvector of $\mA_-$ with eigenvalue $\mu(\Lambda)$, both 
analytic in $\Lambda$, 
then there exists a solution of \eqref{gfirstorder} of form
\begin{equation}
 \WW(\Lambda, x) = V (x,\Lambda ) e^{\mu(\Lambda) x},
\end{equation}
where $V$ is $C^{1}$ in $x$ and locally analytic in $\Lambda$ and,
for any fixed $\btheta < \theta$, satisfies 
\begin{equation}
\label{3.6g}
V(x,\Lambda )=  V^-(\Lambda ) + O (e^{-\bar \theta|x|}|V^- (\Lambda)|),\quad x < 0.
\end{equation}
\end{lem}

\begin{proof}
Setting $\WW(x)=e^{\mu x}V(x)$, we may rewrite $\WW'=\mA \WW$ as 
\begin{equation}
\label{3.7g}
V' = (\mA_- - \mu I)V+\Theta V,
\qquad
\Theta
:= (\mA - \mA_-)=O(e^{-\theta|x|}),
\end{equation}
and seek a solution $V(x,\Lambda )\to V^-(x)$ as $x \to \infty$.
Choose $ \btheta< \theta _1 < \theta $
such that there is a spectral gap 
$|\R \big(\sigma \mA_- - (\mu+\theta_1)\big)|>0$
between $\sigma \mA_-$ and $\mu + \theta_1$.
Then, fixing a base point $\Lambda_0$, we can define on some neighborhood
of $\Lambda_0$ to the complementary $\mA_-$-invariant projections
$P(\Lambda)$ and $Q(\Lambda)$ where $P$ projects onto the direct sum
of all eigenspaces of $\mA_-$ with eigenvalues $\Tmu$ 
satisfying 
$ \R(\Tmu) < \R(\mu) + \theta_1, $
and $Q$ projects onto the direct sum of the remaining eigenspaces,
with eigenvalues satisfying 
$ \R(\Tmu)  > \R(\mu) +\theta _1.  $
By basic matrix perturbation theory (eg. [Kat]) it follows that 
$P$ and $Q$ are analytic in a neighborhood of $\Lambda_0$,  
with
\begin{equation}
\label{3.10g}
\left|e^{(\mA_- - \mu I)x} P \right| 
\le C (e^{\theta_1 x}),
\quad x>0, 
\qquad
\left|e^{(\mA_- - \mu I)x} Q \right| 
\le C (e^{\theta_1 x}), 
\quad x<0.
\end{equation}

It follows that, for $M>0$ sufficiently large, the map $\CalT$ defined by 
\begin{equation}
\label{3.11g}
\begin{aligned}
\CalT V(x) 
&= V^- + \int^x_{-\infty} e^{(\mA_- - \mu I)(x-y)} P
\Theta (y) V(y) dy \\
&\quad - \int^{-M}_x e^{(\mA_- - \mu I)(x-y)} Q \Theta (y) V(y) dy
\end{aligned}
\end{equation}
is a contraction on $L^\infty(-\infty, -M]$.
For, applying \eqref{3.10g}, we have
\begin{equation}
\label{3.12g}
\begin{aligned}
\left|\CalT V_1 - \CalT V_2 \right|_{(x)} 
&\le C |V_1 - V_2|_\infty 
\bigg(\int^x_{-\infty} e^{\theta_1(x-y)} e^{\theta y} dy 
+ \int^{-M}_{x} e^{\theta _1(x-y)}e^{\theta y}dy\bigg)\\
%
&=C |V_1 - V_2|_\infty 
\frac{e^{\theta_1 x}e^{-(\theta-\theta_1)M}}
{\theta-\theta_1}
< \frac{1}{2} |V_1 - V_2|_\infty.
\end{aligned}
\end{equation}

By iteration, we thus obtain a solution $V \in L^\infty (-\infty, -M]$ of $V = 
\CalT V$ with $V\le C_3|V^-|$; since $\CalT$ clearly preserves analyticity
$V(\Lambda, x)$  is 
analytic in $\Lambda$ as the uniform limit of analytic 
iterates (starting with $V_0=0$). 
Differentiation shows that $V$ is a bounded solution of
$V=\CalT V$ if and only if it is a bounded solution of \eqref{3.7g}.
Further, taking $V_1=V$, $V_2=0$ in \eqref{3.12g}, we obtain from
the second to last inequality that 
\begin{equation}
\label{3.13}
|V-V^-| = |\CalT(V) - \CalT(0)| \le C_2  e^{\btheta x} |V| 
\le  C_4e^{\btheta x}|V^-|,
\end{equation}
giving \eqref{3.6g}.
Analyticity, and the bounds \eqref{3.6g},  
extend to $x<0$ by standard analytic dependence for the initial value
problem at $x=-M$. 
\end{proof}

\begin{rem}\label{gapexplanation}
The title ``gap lemma'' alludes to the fact that we
do not make the usual assumption of
a spectral gap between $\mu(\Lambda)$ 
and the remaining eigenvalues of $\mA_-$, as
in standard results on asymptotic behavior of ODE \cite{Co};
that is, the lemma asserts that exponential
decay of $\mA$ can substitute for a spectral gap.
\end{rem}

\br\label{bettertheta}
In the case $\Re \sigma(\mA_+) >-\theta$, we may take $Q=\emptyset$ 
and $\tilde \theta=\theta$, improving \eqref{3.6g}.
\er
\medbreak

\begin{cor}[{The conjugation lemma \cite{MeZ}}]
\label{conjugation}
Given (h0), there exist locally to any given $\Lambda_0\in \Omega $
invertible linear transformations $P_+(x,\Lambda)=I+\Theta_+(x,\Lambda)$ and 
$P_-(x,\Lambda) =I+\Theta_-(x,\Lambda)$ defined
on $x\ge 0$ and $x\le 0$, respectively,
$\Phi_\pm$ analytic in $\Lambda$ as functions from $\Omega$
to $C^0 [0,\pm\infty)$, such that: 
\medbreak
(i)
For any fixed $0<\btheta<\theta$ and $0\le k\le K+1$,  $j\ge 0$, 
\begin{equation}
\label{Pdecay} 
|(\partial/\partial \Lambda)^j(\partial/\partial x)^k
\Theta_\pm |\le C(j,k)  e^{-\tilde \theta |x|}
\quad
\text{\rm for } x\gtrless 0. 
\end{equation}
\smallbreak
(ii)  The change of coordinates $\WW=:P_\pm \ZZ$,
$\FF=: P_\pm \GG$ reduces \eqref{gfirstorder} to 
\begin{equation} \label{glimit}
\ZZ'-\mA_\pm \ZZ = \GG
\quad
\text{\rm for } x\gtrless 0.
\end{equation}
\end{cor}

\br\label{conjrmk}
Equivalently, solutions of \eqref{gfirstorder} may be factored as 
\begin{equation}
\label{Wfactor}
\WW=(I+ \Theta_\pm)\ZZ_\pm, 
\end{equation}
where $\ZZ_\pm$ satisfy the limiting, constant-coefficient 
equations \eqref{glimit} and $\Theta_\pm$ satisfy \eqref{Pdecay}. 
\er

\begin{proof}
Substituting $\WW=P_- Z$ into \eqref{gfirstorder}, equating
to \eqref{glimit}, and rearranging, we obtain the defining equation
\begin{equation}
\label{matrixODE}
P_-'= \mA_-P_- - P_- \mA, \qquad P_- \to I \quad \text{\rm as}
\quad
x\to -\infty.
\end{equation}
Viewed as a vector equation, this has the form
$ P_-'=\CalA P_-, $
where $\CalA$ approaches exponentially as $x\to -\infty$ to its
limit $\CalA_-$, defined by
\begin{equation}
\label{CalA}
\CalA_- P:= 
\mA_-P- P \mA_-.
\end{equation}
The limiting operator $\CalA_-$ evidently has 
analytic eigenvalue, eigenvector pair $\mu\equiv 0$, $P_-\equiv I$,
whence the result follows by Lemma \ref{gaplemma} for $j=k=0$.
The $x$-derivative bounds $0<k\le K+1$ then follow from the ODE and its first
$K$ derivatives,
and the $\Lambda$-derivative bounds from 
standard interior estimates for analytic functions.
Finally, invertibility of $P_-$ follows for $x$ large and negative
from \eqref{Pdecay}
and for $x\le 0$ by global existence of a solution to
$$
(P^{-1})'=-P^{-1}P'P^{-1}
=A_+P^{-1}-P^{-1}A.
$$
A symmetric argument gives the result for $P_+$.

\end{proof}

\section{A numerical gap lemma} \label{numgap}

We now establish the main result of the paper, 
a numerical analog of Lemma \ref{gaplemma}.

\subsection{Continuous problem}\label{contprob}
Consider similarly as in \eqref{gfirstorder} a first-order ODE 
\begin{equation}
\label{nsys}
W'-A(x)W=0,
\quad
W\in \CC^N
\end{equation}
on the half-line $x\in [0,+\infty)$,
assuming exponential convergence
\begin{equation}
\label{expdecay3}
|(\partial/\partial x)^k(A- A_+)| 
\le Ce^{-\theta x }, \, 
\quad
\text{\rm for } x\ge 0, \, 0\le k\le K,
\end{equation}
$C$, $\theta>0$, of $A$ to $A_+$ and asymptotic stationarity
\be\label{kercond}
V_+ \in \kernel A_+.
\ee
Suppose further that there holds the following {\it gap condition}.
\medskip

\begin{ass}\label{gapassum}
$\Sigma_+$ is a $k$-dimensional
invariant subspace of $A_+$
containing all eigenmodes associated with nonnegative real part
eigenvalues, and no eigenmodes with real part $\le -\theta$,
with associated eigenprojection $\Pi_+$.
\end{ass}

\noindent
Let $\Pi_0$ be an arbitrary projection of rank $(n-k)$.

\bl\label{exist}
For generic $\Pi_0$, specifically those satisfying \eqref{lop} below, 
there exists for any $\alpha \in \Range \Pi_0$ 
a unique solution of \eqref{nsys} under the projective boundary conditions
\be\label{nbc}
\Pi_0 W(0)= \alpha, \quad \lim_{x\to +\infty} 
\Pi_+(W(x)-V_+)=0,
\ee
satisfying for any $0<\tilde \theta<\theta$
\be
|(\partial/\partial x)^k
(W(x)-V_+)| \le C(\theta, \tilde \theta) e^{-\tilde \theta x }, \, 
\quad
\text{\rm for } x\ge 0, \, 0\le k\le K.
\ee
\el

\begin{proof}
By Lemma \ref{conjugation}, there exists an invertible 
coordinate transformation
\be \label{factor}
W=(I+ \Theta)Z
\ee
with
\begin{equation}
\label{tdecay} 
| (\partial/\partial x)^k
\Theta |\le C(k) e^{-\tilde \theta |x|}
\quad
\text{\rm for } x\ge 0,
\end{equation}
converting \eqref{nsys}, \eqref{nbc} to
\begin{equation} \label{glimit2}
Z'-A_+ Z = 0
\quad
\text{\rm for } x\ge 0
\end{equation}
and
\be\label{znbc}
\tilde \Pi_0 Z(0)= \tilde \alpha,
\quad 
\lim_{x\to +\infty} 
\Pi_+(Z(x)-V_+)=0,
\ee
where $\tilde \Pi_0:= (I+\Theta(0))^{-1} \Pi_0 (I+\Theta(0))$ 
is again arbitrary and $\tilde \alpha:= (I+\Theta(0))^{-1}\alpha$.
By inspection, there exists a unique solution if and only
if $\kernel \tilde \Pi_0$ is transverse to $\Sigma_+$, i.e.,
there holds the generically satisfied Evans/Lopatinski condition 
\be\label{lop}
 \det \tilde \Pi_0 \tilde \Sigma_+\ne 0 , 
\ee
where $\tilde \Sigma_+$
is the complementary (rank $n-k$) invariant subspace of $\Sigma_+$,
consisting of the sum of the constant solution $Z\equiv V_+$
and the solution of the Cauchy problem $Z(0)=\tilde \alpha- \tilde \Pi_0 V_+$
under the flow of the constant-coefficient equation
\eqref{glimit2} restricted to the subspace $\tilde \Sigma_+$
of exponentially decaying solutions.
\end{proof}

\subsection{Discrete problem}\label{discretized}
We now consider a discretized version of \eqref{nsys}, \eqref{nbc}
on the truncated domain $x\in [0,M]$ with the corresponding
projective boundary conditions
\be\label{dbc}
\Pi_0 W(0)= \alpha, \quad \Pi_+(W(M)-V_+)=0,
\ee
and examine convergence as $M\to +\infty$ and mesh size goes
to zero of the approximate to the exact solution.

\br\label{keyrmk}
In view of the discussion of the introduction, our main
interest in the context of Evans function computations is
the ``Dirichlet'' case $\Sigma_+=\CC^n$ arising for one-sided schemes
suitable for shooting methods: in particular, 
the centered exterior-project method or
the polar coordinate method linearized about a desired exact solution.
In this more favorable case, there is no boundary condition at $x=0$,
and no arbitrary projection $\Pi_0$, and the projective boundary
conditions \eqref{dbc} reduce to the simple Dirichlet condition
$W(M)=V_+$.
\er

\subsubsection{Difference scheme}\label{diffscheme}
We assume a general linear difference scheme of form
\be\label{diff}
\cS \cW= 0,
\ee
with boundary conditions
\be\label{dbc2}
\Pi_0 \cW_0= \alpha, \quad \Pi_+(\cW_J-V_+)=0,
\ee
where $\cW=(\cW_1, \dots, \cW_J)$ are approximations of $W$ at mesh points
$x_j$, $j=0,\dots, J$ and $\cS$ is a linear difference operator with
finite stencil $\{j-\ell, j+\ell\}$, i.e., the value of $(\cS \cW)_j$ depends
on $\cW$ only through $\{\cW_{j-\ell}, \dots, \cW_{j+\ell}\}$.
Moreover, we assume that $(\cS \cW)_j$ depends on $A$ linearly and only
through the restriction of $A$ to the interval $[x_{j-\ell}, x_{j+\ell}]$.
We define boundary and truncation errors $\eps_0$, $\eps_J$ and 
$\tau=(\tau_0, \dots, \tau_J)$ as usual by
\ba\label{truncdef}
\cS \bar \cW= h_j\tau,\quad
\Pi_0 \bar \cW_0-\alpha = \eps_0 , \quad
\Pi_+(\bar \cW_J-V_+)= \eps_J,
\ea
where $\bar \cW_J:=W(x_j)$ denotes the solution of the continuous problem
\eqref{nsys}, \eqref{nbc} on $[0,+\infty)$ sampled at mesh points $x_j$,
and $h_j$ is the $j$th mesh length, defined as the maximum difference
between points $x_k$ involved in the evaluation of $\cS_j$.

This could be a boundary-value scheme, with $\cS$
realized as a large $Jn\times Jn$ banded matrix.
Or, in the case of our main interest $\Sigma=\CC^n$ 
that boundary conditions are imposed
only at one end $x=M$, it could be a backward Cauchy
solver such as RK45, with $\cS$ realized 
as a series of $J-\ell$ successive
$(2\ell+1)n\times (2\ell+1)n$ matrix multiplications.

\subsubsection{Basic assumptions}\label{diffassumptions}
We do not specifiy the details of the scheme, 
other than to make 
certain mild assumptions.

\begin{assums}\label{stabassums}
{ }

{\it (i) $k$-th order consistency}: As mesh size $h_j\to 0$, 
\be\label{trunc}
|\tau_j|\le C
h_j^k \sup_{z\in[x_{j-\ell}, x_{j+\ell}]}|\partial_x^{k+1}W(z)| \to 0.
\ee

{(ii) \it Strict constant-coefficient stability}: 
In the {constant-coefficient case} $A\equiv A_+$,
if $\Re \sigma(A_+)\le \mu$, then, for any $\tilde \mu>\mu$, 
Cauchy information
$\cS\cW=\tau h$ and $\cW_0=\eps_0$ imply 
(under appropriate mesh restrictions)
\be\label{stabassumption}
|\hat \Delta^r \cW_j|\le C\sup_{j-r\le m\le j} |h_m|^r \big(
\eps_0 e^{\tilde \mu x_j} 
+ \sum_{0\le k\le j}  e^{\tilde \mu (x_j-x_k)}\tau_k h_k \big)
\ee
for $0\le r\le 2$,
where $\hat \Delta$ is the backward difference operator
$(\hat \Delta f)_j:=f_j-f_{j-1}$.
Likewise, if $\Re \sigma(A_+)\ge \mu$, 
then, for any $\tilde \mu< \mu$, 
$\cS\cW=\tau h$ and $\cW_J=\eps_J$ imply 
\be\label{stabassumption2}
|\hat \Delta^r \cW_j|\le C \sup_{j-r\le m\le j}|h_m|^r \big(
\eps_J e^{\tilde \mu (x_j-x_J)} 
+ \sum_{j\le k\le J}  e^{\tilde \mu (x_j-x_k)}\tau_k h_k \big).
\ee
In the case $\Sigma=\CC^n$ of our main interest, 
we require only \eqref{stabassumption2}
and only for $\mu\le 0$.\footnote{
This is all that is needed in any case for \eqref{stabassumption2},
but this seems to give no advantage in the general case since
there remains the more restrictive assumption \eqref{stabassumption}
in the other direction.
} 
\end{assums}

As is customary, we ignore machine error.

\br\label{validity}
Condition (i) is of course standard.
In the situation of our main interest of a backward Cauchy solver 
in the case $\Sigma=\CC^n$, condition (ii)
(in this case, the single condition \eqref{stabassumption2}) 
for $r=0$ is equivalent
by Duhamel's principle/variation of constants to the homogeneous stability
condition
$$
|\cW_j|\le C \eps_k e^{\tilde \mu (x_j-x_k)}
$$ 
for $x_j<x_k$, $\tilde \mu\le 0$, when $\cS\cW=0$ and $\cW_k=\eps_k$, 
which is related to a circle of ideas including 
$A$-stability and one-sided Lipshitz bounds \cite{D1,D2,HNW1,HNW2}
regarding
approximation of stiff ODE.
For general schemes it may impose impractically severe
limitations on the mesh size; however, it is satisfied 
without such
conditions for 
RK45 and other ``nice'' schemes, 
as may be seen by applying a constant linear
coordinate transformation $A \to \tilde A:=SAS^{-1}$
for which $\Re \tilde A:=(1/2)(\tilde A+ \tilde A^*)\ge \check \mu>\mu$
with $\check \mu$ arbitrarily close to $\mu$, then applying the results
of \cite{D1,D2,HNW1,HNW2}.
For some simple examples, see Remark \ref{egstab}.
Likewise, for Cauchy solvers (ii) ($r=1,2$) follows 
using the difference scheme from (ii) ($r=0$).
For general boundary-value schemes, (ii) must be checked on
a scheme-by-scheme basis.
For general theory on  
stability of boundary-value schemes, see, e.g., \cite{Kr,Be1,Be2}. 
\er

\br\label{dualrmk}
Evidently, (ii) implies also the dual bounds
\be\label{dual1}
|\cW_j|\le C \big( \eps_0 e^{\tilde \mu x_j} 
+ 
\sup_{j-r\le m\le j} |h_m|^r 
\sum_{0\le k\le j}  e^{\tilde \mu (x_j-x_k)}\tau_k h_k \big)
\ee
and
\be\label{dual2}
|\cW_j|\le C 
\big( \eps_J e^{\tilde \mu (x_j-x_J)} 
+ 
\sup_{j-r\le m\le j}|h_m|^r 
\sum_{j\le k\le J}  e^{\tilde \mu (x_j-x_k)}\tau_k h_k \big)
\ee
for differenced data $\cS\cW=\Delta^r(\tau h)$, $0\le r\le 2$,
where $\Delta$ denotes the forward difference operator
$(\Delta f)_j:=f_{j+1}-f_{j}$.
\er

\subsection{Discrete conjugation error} \label{discon}

We now make the key observation that 
{\it the conjugating transformation
for the continuous problem, up to a small commutation error,
is a conjugator for the discrete problem as well.}

\begin{example}\label{Eulereg}
\textup{
Consider the first-order (forward explicit/backward implicit)
Euler scheme
$$
(\cS\cW)_j:=\cW_{j+1}- \cW_j -  h_j A_j \cW_j=0.
$$
Substituting $\cW_j=:P_j \cZ_j$, $P_j:=P(x_j)$, where $P=I+\Theta$ is
the continuous conjugator of \eqref{factor},
we obtain after a brief calculation
$$
(\tilde \cS\cZ)_j:=\cZ_{j+1}- \cZ_j -  h_j A_{+j} \cZ_j=
(\Psi \cZ)_j,
$$
where $\tilde S$ is the realization of the same first-order
Euler scheme to the constant-coefficient case $A\equiv A_+$ and
$\Psi$ is the commutator error
$$
(\Psi \cZ)_j):= -h_jP_j^{-1}\Big(P_{j+1}-P_j -h_j(A_jP_j-P_{j+1}A_{+j})\Big)\cZ_j.
$$
Noting that
$ \cP_{j+1}-\cP_j = h_j(A_j\cP_j-\cP_{j+1}A_{+j}) $
is a discretization of the ODE
$P'=AP-PA_+$ defining $P$, \eqref{matrixODE},
we find similarly as in \eqref{trunc} that  the exact solution
$P$ has truncation error
$$
\hat \tau_j:=
 h_j^{-1}\Big(P_{j+1}-P_j - h_j(A_jP_j-P_{j+1}A_{+j}) \Big)
= O(|h_j| \sup_{x\in [x_{j-\ell}, x_{j+\ell}]}|P''|)
=O(|h_j| e^{-\tilde \theta x}),
$$
we find that 
\be\label{lastest}
|(\Psi \cZ)_j|\le 
C |h_j|^2e^{-\tilde \theta x_j} |\cZ_j|.
\ee
}
\end{example}

\br
Note that the righthand side of \eqref{lastest} is smaller by a factor $h_j$ 
than the corresponding estimate
$$
(\tilde S\cW)_j=O(|h_j|\sup_{j-\ell\le m\le j+\ell}|A_m-A_{m+}|)
\sup_{j-\ell\le m\le j+\ell}|\cW_m|
\le C |h_j|e^{-\theta x_j} \sup_{j-\ell\le m\le j+\ell} |\cW_m|.
$$
obtained by separating out constant-coefficient and exponentially
decaying parts in the original equation for $\cW$.
This additional factor is crucial in the contraction
argument of \S \ref{numconv}.
\er

\begin{example}\label{Mideg}
\textup{
Consider the second-order (forward/backward implicit) midpoint scheme
\be\label{mid}
(\cS\cW)_{j+1}:=\cW_{j+1}- \cW_{j-1} - (x_{j+1}-x_{j-1}) A_j \cW_j=0.
\ee
Substituting $\cW_j=:P_j \cZ_j$, $P=I+\Theta$ as in \eqref{factor},
we obtain 
$$
(\tilde \cS\cZ)_j:=\cZ_{j+1}- \cZ_{j-1} - (x_{j+1}-x_{j-1}) A_{+j} \cZ_j=
(\Psi \cZ)_j
+\Delta(\Psi^1 \cZ)_j
$$
where $\tilde S$ is the constant-coefficient realization of $\cS$,
$\Delta$ is the forward difference operator $(\Delta f)_j:=f_{j+1}-f_j$, and
$$
\begin{aligned}
(\Psi \cZ)_j)&:= -(x_{j+1}-x_{j-1})
P_j^{-1}\Big(P_{j+1}-P_{j-1} -(x_{j+1}-x_{j-1})(A_jP_j-P_{j+1}A_{+j})\Big)\cZ_j\\
&\quad +
P_j^{-1}\Big(\frac{P_{j+1}+P_{j-1}}{2}-P_j\Big)(x_{j+1}-x_{j-1})A_j \cZ_j\\
&\quad +
\Big(\frac{1}{2}\Big)
\Delta^2\Big(P_j^{-1}(P_{j+1}-P_{j-1}) \Big) \cZ_{j-1},\\
(\Psi^1 \cZ)_j&:=
\Big(P_j^{-1}(P_{j+1}-P_{j-1}) \Big)(\cZ_{j}- \cZ_{j-1}\Big).\\
\end{aligned}
$$
Arguing as in the previous example, we obtain similarly as in
\eqref{lastest} the bound
\be\label{lastestmid}
|(\Psi \cZ)_j|\le 
C |h_j|^2e^{-\tilde \theta x_j} \sup_{j-1\le m\le j}|\cZ_m|.
\ee
and, likewise,
\be\label{ho}
|(\Psi^1 \cZ)_j|\le C |h_j|e^{-\tilde \theta x_j} \sup_{j-1\le m\le j}|\cZ_m|.
\ee
}
\end{example}

\br\label{egstab}
The implicit backward Euler scheme of Example \ref{Eulereg}
is $A$-stable as a Cauchy solver in the backward $x$ direction,
hence satisfies Assumption \ref{stabassums} in the 
one-sided case $\Sigma_+=\CC^n$, $\tilde \mu=0$,
with no assumption on the mesh size.
For example, in the scalar case 
\be\label{scalar}
W'=-aW, \quad a>0,
\ee
$\cW_{j+1}=\cW_{j}+h_ja \cW_j$ yields immediately
$\cW_j= (I+h_ja)^{-1}\cW_{j+1}\le \cW_{j+1}$, for any $h_j>0$.
The implicit Midpoint method viewed as a two-step backward scheme
has characteristic roots $-ah_j\pm \sqrt{(ah_j)^2+1}$ that 
remain strictly $\gtrless 1$ independent of $h_j>0$, but is
not backward $A$-stable as a Cauchy solver.
\er

Generalizing the results of the examples, we make the
following final assumption on the scheme, which appears
to be satisfied in most if not all cases of interest.

\begin{ass}\label{zcon}
Under the change of coordinates $\cW_j=:P_j \cZ_j$,
$P$ defined as in \eqref{factor},
equation \eqref{nsys} becomes
\be\label{cZeq}
\tilde S \cZ= 
\Psi^0 \cZ + \Delta (\Psi^1 \cZ),
\ee
where $\tilde S$ is the same discretization scheme used for $\cW$
applied to the constant-coefficient case $A\equiv A_+$, $\Delta$
is the forward difference operator, and
\be\label{Psiest}
|(\Psi^r \cZ)_j|\le C(\theta,\tilde \theta) 
|h_j|^{1-r}e^{-\tilde \theta x_j} \sup_{j-\ell-1\le m\le j+\ell}
|\cZ_m|.
\ee
\end{ass}

\subsection{Numerical convergence lemma} \label{numconv}

With these preparations, it is straightforward to establish our 
following main result.

\bt\label{convthm}
Assuming \eqref{expdecay3}, \eqref{kercond}, \eqref{lop}, 
and Assumptions \ref{gapassum}, \ref{stabassums}, and \ref{zcon},
for fixed $0<\tilde \theta<\theta$,
for $\sup_j |h_j| $ sufficiently small and $M>0$ sufficiently large,
there exists a unique solution $\cW$ of \eqref{diff}, \eqref{dbc2},
which, moreover, satisfies 
for $C>0$ independent of $M$, $\tau$,
\be\label{convrate}
|\cW_j-\bar W(x_j)|\le 
C(e^{-\tilde \theta M } 
+ \sup_j|\tau_j|)
\quad \hbox{\rm for all } \, j=0, \dots, J,
\ee
where $\bar W$ is the solution of \eqref{nsys}, \eqref{nbc}
guaranteed by Lemma \ref{exist}, 
and $\theta$ is as in \eqref{expdecay3}.
For $\Sigma_+=\CC^n$, 
the mesh condition $\sup_j |h_j|<<1$ can be relaxed to
$\sup_j |h_j| e^{(-\tilde \theta-\tilde \nu) x_j}<<1$,
where $-\tilde \theta < \tilde \nu \le 0$ is less 
than the smallest real part of the eigenvalues of $A_+$. 
\et

That is, we assert existence and convergence so long as the spectral
gap $\min \Re \sigma \big(A_+|_{\Sigma_+} \big)$ associated with 
the boundary projection $\Pi_+$ at $+\infty$
is greater than  $-\theta$, where $\theta$ as in \eqref{expdecay3}
is the exponential rate of convervence of $A$ to $A_+$
as $x\to +\infty$.
On the other hand, standard theory \cite{Be1} requires
a positive spectral gap $\beta$ and concludes convergence at
rate $Ce^{-\tilde \beta M}$ for $0<\tilde \beta<\beta$.
In other words, though it fails the positive gap condition of 
standard theory, the problem remains numerically well-conditioned so
long as ``badness'' of the boundary condition
as measured by negativity of the spectral gap is less than
the rate of exponential convergence, in exact analogy with the
gap lemma of continuous theory \cite{GZ,KS,ZH}.

\begin{proof}
By Assumption \ref{zcon},
$\tilde S \cZ= \sum_{r=0}^2 \Delta^r \Psi^r \cZ$ and
$\tilde S \bar \cZ= \sum_{r=0}^2 \Delta^r \Psi^r \bar \cZ + \cF$,
where $\cW_j=:P_j\cZ_j$, 
$\bar \cW_j=\bar W(x_j)=:P_j\bar \cZ_j$, 
and $h_j\tau_j=:\cF_j$.  Defining the convergence error 
$\cE:=\cZ-\bar \cZ$, we thus obtain the error equation
\be\label{erreq}
\tilde S \cE= \sum_{r=0}^2 \Delta^r \Psi^r \cE - \cF,
\ee
with boundary conditions
\be\label{errbc}
\tilde \Pi_0 \cE_0= 0 \quad  \hbox{\rm and }\quad
\tilde \Pi_+(\cE_J)=
-\tilde \Pi_+(\bar \cZ_J-\tilde V_+) =O(e^{-\tilde \theta M}).
\ee

Denoting by $\check \cE=\cT(\cE)$ the solution of 
$\tilde S \check \cE= \sum_{r=0}^2 \Delta^r \Psi^r \cE - \cF$,
\eqref{errbc}, we thus obtain the fixed-point formulation
\be\label{fixedpt}
\cE= \cT(\cE)
\ee
for the solution of \eqref{erreq}--\eqref{errbc}, and
thus for the solution $\cZ=\bar \cZ+\cE$ of the original problem.

Applying bounds \eqref{dual1}-\eqref{dual2} of Remark \ref{dualrmk},
together with \eqref{Psiest}, \eqref{errbc}, and \eqref{trunc}, and
the principle of linear superposition, we obtain
\ba\label{keycalc}
|\check \cE_j|&\le C \big( \eps_0 e^{\tilde \mu x_j} 
+ 
\sum_{0\le k\le j}  e^{\beta (x_j-x_k)}
\big( |(\Psi^0 \cE)_k| +\sup_{j-1\le m\le j} |h_m| |(\Psi^1 \cE)_k| \big)\\
&\quad +\sum_{0\le k\le j}  e^{\tilde \mu (x_j-x_k)}\tau_k h_k 
\big)\\
&\quad +
C \big( \eps_J e^{\tilde \mu (x_j-x_J)} 
+ 
\sum_{j\le k\le J}  e^{\nu (x_j-x_k)} 
\big( |(\Psi^0 \cE)_k| +\sup_{j-1\le m\le j} |h_m| |(\Psi^1 \cE)_k| \big)\\
&\quad +
\sum_{j\le k\le J}  e^{\tilde \mu (x_j-x_k)}\tau_k h_k \big),
\ea
$\Psi:=(\Psi^0,\Psi^1,\Psi^2)$, where $\beta<0$ is greater than
the largest real part of the eigenvalues of $A_+$ associated
with the invariant subspace $\tilde \Sigma_+$ complementary
to $\Sigma_+$, and $-\tilde \theta < \tilde \nu<\nu \le 0$ is less that
than the smallest real part of the eigenvalues of $A_+$ associated
with $\Sigma_+$; see Assumption \ref{gapassum}.
From \eqref{keycalc}, we readily obtain by a discrete version of 
calculation \eqref{3.12g} in the argument of Lemma \ref{gaplemma} the
a priori estimate
\be\label{convrateE}
\sup_{0\le j\le J}|\cE_j|\le 
C\big(e^{-\tilde \theta M } + \sup_j|\tau_j| +\sup_j|h_j|
\sup_{0\le j\le J}|\cE_j|
\big),
\ee
which for $\sup_j |h_j|$ sufficiently small implies
\be\label{fconvrateE}
\sup_{0\le j\le J}|\cE_j|\le 
C\big(e^{-\tilde \theta M } + \sup_j|\tau_j| \big),
\ee
hence, by boundedness of $|P_j|$ and $\cW_j-\bar W(x_j):=P_j\cE_j$,
the desired bound \eqref{convrate}.
A similar estimate yields contractivity of $\cT$ in the sup-norm, hence
existence and uniqueness in the class $\ell^\infty[0,J]$.

Finally, observing in case $\Sigma_+=\CC^n$ that
only the $\sum_{j\le k\le J}$ terms in \eqref{keycalc} appear,
we may bound
$\sup_{j-1\le m\le j} |h_m| |(\Psi^1 \cE)_k|$ using $j\le k$ by
$$
\sup_{j-1\le m\le j} |h_m| 
\sup_{k-\ell-1\le m\le k+\ell} e^{-\tilde \theta x_m}|h_m|
\le
\sup_{0\le j\le J}\big(e^{(-\tilde \theta-\tilde \nu)x_j}|h_j|\big)
\sup_{k-\ell-1\le m\le k+\ell} e^{\tilde \nu x_m}|h_m|
$$
and similarly for $|(\Psi^0 \cE)_k|$,
to obtain by the same argument
$$
\sup_{0\le j\le J}|\cE_j|\le 
C\big(e^{-\tilde \theta M } + \sup_j|\tau_j|
+\sup_{0\le j\le J}\big(e^{(-\tilde \theta-\tilde \nu)x_j}|h_j|\big)
\sup_{0\le j\le J}|\cE_j|
\big),
$$
in place of \eqref{convrateE}, yielding existence, uniqueness, and
convergence under the relaxed mesh condition
$\sup_j |h_j| e^{(-\tilde \theta-\tilde \nu) x_j}<<1$.
\end{proof}

\br\label{apps}
By Remark \ref{egstab}, Theorem \ref{convthm} applies in
particular to the one-sided, or Dirichlet, 
case $\Sigma_+=\CC^n$, with
the implicit backward Euler method of Example \ref{Eulereg}.
\er

\br\label{largestep}
Since $-\tilde \theta -\tilde \nu<0$ by Assumption \ref{gapassum},
in the case $\Sigma_+=\CC^n$ of our main interest,
the mesh and truncation error conditions
$$
\sup_j |h_j| e^{(-\tilde \theta-\tilde \nu) x_j}<<1
\quad \hbox{\rm and }\quad
\sup_j |\tau_j|\le \sup_j |h_j|^k e^{-\tilde \theta x_j}<<1
$$
required for convergence allow for arbitrarily large step 
size as $x_j\to +\infty$, hence our result indeed applies to
shooting-type schemes with adaptive step size such as 
are used in practice.
\er

\br\label{pfrmk}
The method of proof in the argument of Theorem \ref{numconv},
based on Lemma \ref{conjugation},
is designed to handle the general boundary-value case.
For shooting methods, there is a much simpler
proof on $[L,M]$ with $1<<L\le M$, as in the proof of Lemma \ref{gaplemma},
just separating constant-coefficient and exponentially decaying parts 
of $A$ in the original $\cW$ equation and obtaining contraction
as in the proof of Lemma \ref{gaplemma} from largeness of $L$.
The extension to the full domain $[0,M]$ then follows similarly
as in the continuous case by standard convergence results 
for the Cauchy problem on a bounded domain.
The conjugation argument is used to obtain contraction for
globally defined schemes on the full domain $[0,M]$.
\er


\subsection{Convergence of the centered exterior product method} \label{extstab}

The centered exterior product method as described in
the introduction consists of solving
\be\label{exeq}
\WW'=(\mA(x,\lambda)-\mu_{S_+})\WW,
\ee
from $x=+\infty$ to $x=0$ for
$\WW^+:=W_1^+\wedge\cdots\wedge W_k^+$,
where the linear operator $\mA$ is defined
by the Leibnitz formula
\be\label{Leib}
\mA W_1\wedge\cdots\wedge W_k:=
(A W_1\wedge W_2\wedge \cdots\wedge W_k )+
\dots +
 (W_1\wedge \cdots\wedge AW_k)
\ee
with eigenvectors and eigenvalues $\cR=r_1\wedge \cdots \wedge r_k$,
$\mu=a_1+ \dots + a_k$, where $r_j$ and $a_j$ are eigenvectors of $A$,
and $\mu_{S_+}$ is the sum of the eigenvalues of $A_+$ associated
with the $k$-dimensional stable subspace $S_{+}$, to obtain
a solution asymptotic to an eigenvector $\cR_{S_+}$ of $\mA_+$
obtained as the wedge product of a basis of $S_+$;
solving the symmetric equation from $x=-\infty$ to $x=0$ for a solution
asymptotic at $-\infty$ to an eigenvector $\cR_{S_-}$ of $\mA_-$
associated with the unstable subspace $S_-$ of $A_-$; then
evaluating the Evans function following \eqref{exteval} as
\be\label{deval}
D(\lambda):=\langle (\WW^+\wedge \WW^-)|_{x=0}\rangle,
\ee
$\WW^-:=W_{1}\wedge \cdots \wedge W_{n-k}^-$,
where $\langle \cdot \rangle$ denotes coordinatization 
in the standard Euclidean basis, i.e.
$\eta=:\langle \eta \rangle (e_1\wedge \cdots \wedge e_n)$
for an $n$-form $\eta$.\footnote{See Section \ref{exinit} for
numerical prescriptions of $\cR_{S_\pm}$.}

Equation \eqref{exeq} is of the form \eqref{nsys} considered
in Section \ref{contprob}, so that we may apply our just-developed 
theory. Moreover, as $\mu_{S_+}$ and $\mu_{S_-}$ respectively
are the smallest real part and largest real part eigenvalues
of $\mA_+$ and $\mA_-$, we are in the more favorable 
one-sided, or ``Dirichlet'' case $\Sigma_\pm=\CC^{N_\pm}$ 
suitable for shooting methods, where $N_+:={n\choose k}$ and
$N_-:={n\choose n-k}$ are the dimensions of the systems for
$\WW^+$ and $\WW^-$.
We may thus approximate the solution as described in Section
\ref{discretized} by solving a pair of finite difference schemes
\eqref{diff}, on $[0,M]$ and $[0,-M]$, respectively,
discretized as $0=x_0<x_1<\cdots<x_J=M$ and
$0=x_0>x_{-1}> \cdots> x_{-J}=-M$,
with Dirichlet boundary conditions
\be\label{exbc}
\cW_J^+=\cR_{S_+}, \quad
\cW_{-J}^-=\cR_{S_-},
\ee
determining thereby a numerically approximated Evans function
\be\label{nevans}
\cD^{M,h}(\lambda):=
\langle (\cW_0^+\wedge \cW_0^-)\rangle,
\ee
where $h:=(h_{-J},\dots, h_J)$ is the vector of mesh blocks
used to discretize $[-M,M]$.

In the above discussion, we have implicitly assumed the 
{\it consistent splitting hypothesis} of \cite{AGJ}: 
that the dimensions of the stable subspace of $A_+(\lambda)$
and the unstable subspace of $A_-(\lambda)$ 
sum to full rank $n$ on the subset of
$\lambda$ under consideration.
By standard considerations \cite{He,GZ,Z1},
the ``region of consistent splitting'' on which this holds
typically includes the component of real $+\infty$ in the complement of the
essential spectrum of the associated linearized differential operator $L$
of \eqref{eval} whose point spectra the Evans function is designed to determine.
However, as pointed out in \cite{GZ,ZH}, it is sometimes useful to 
extend this region of investigation and study also eigenvalues embedded
in the essential spectrum.

Following \cite{GZ,ZH}, we thus consider the problem
in a slightly more general setting, substituting in place
of consistent splitting the following {\it gap assumption}.

\begin{ass}\label{exgap}
\textup{
On $\Lambda \subset \CC$, 
the spaces $S_+$ and $S_-$ are invariant subspaces
of $A_+$ and $A_-$, analytic in $\lambda$, with
dimensions $k$ and $(n-k)$.
Moreover, the spectral gaps $\nu_+$ and $\nu_-$ defined as the
maximum of the difference between the smallest real part of the eigenvalues
of $A_+$ not associated with $S_+$ and the largest real part of those
associated with $S_+$ and
the maximum of the difference between the 
smallest real part of the eigenvalues of $A_-$ associated with $S_-$ and
the largest real part of those not associated with $S_+$, satisfy
\be\label{exgapeq}
\nu_j > -\theta, \quad j=\pm,
\ee
where $\theta>0$ as in \eqref{expdecay3} is the exponential rate of convergence
of $A$ to $A_\pm$ as $x\to \pm \infty$.
}
\end{ass}

Applying Theorem \ref{convthm} in this context, we 
obtain the follow convergence result.

\bc\label{exconv}
Under Assumptions \ref{stabassums}, \ref{zcon}, and \ref{exgap},
for fixed $0<\tilde \theta<\theta$,
for $M>0$ sufficiently large and
$\sup_j |h_j| e^{(-\tilde \theta-\tilde \nu) x_j}$ sufficiently small,
where $-\tilde \theta < \tilde \nu \le 0$ is less than $\min \nu_\pm$, 
there exist unique solutions $\cW^\pm$ of 
\eqref{nsys}, \eqref{exbc} determining an approximate 
Evans function $\cD^{M,h}(\lambda)$ satisfying
\be\label{exconvrate}
|\cD^{M,h}-D|\le 
C(e^{-\tilde \theta M } 
+ \sup_j|\tau_j|)
\ee
uniformly on compact subsets of $\Lambda$,
where $D$ is constructed following \eqref{deval}
from the solutions $\WW^\pm$ of \eqref{exeq} 
guaranteed by Lemma \ref{exist}, 
$\theta$ is as in \eqref{expdecay3}, and $\tau_j$ is truncation error.
\ec

\begin{proof}
Immediate, observing that Assumption \ref{exgap} implies
Assumption \ref{gapassum}.
\end{proof}

\subsubsection{Mesh requirements, and 
computations in the essential spectrum} \label{meshreq}

So long as we remain in the region of consistent splitting (see
discussion above Assumption \ref{exgap}), i.e., away from the
essential spectrum of the operator $L$ whose point spectra we
seek to study, we have $\Re \sigma \mA \ge 0$, with a simple
eigenvalue at zero, from which we may obtain the stability
property \eqref{stabassumption2} of (ii) needed for the
convergence proof, with value $\tilde \mu=0$, even though
the statement of (ii) is not strictly satisfied.
Indeed, this situation holds for general difference schemes
in case $\Sigma_+=\CC^n$ whenever $\Re \sigma A_+\ge 0$
and zero is a semisimple eigenvalue of $A_+$, yielding the
same convergence results stated in Theorem \ref{numconv} and
Corollary \ref{exconv}.
Moreover, the spectral gaps $\nu_\pm$ of Assumption \ref{exgap}
are identically zero, and $\tilde \nu$ (by semisimplicity)
may be taken as zero as well.

Together with Remark \ref{validity}, this shows that, for shooting
methods based on an $A$-stable Cauchy solver (e.g., RK45),
there is no requirement on the mesh size $|h_j|$ beyond
the requirement
$\sup_j |h_j| e^{(-\tilde \theta-\tilde \nu) x_j}
\le \sup_j |h_j| e^{-\tilde \theta x_j} $ sufficiently small
stated explicitly in the Theorem (Corollary).
This agrees with observations in \cite{Br,BrZ,HuZ1,BHRZ,HLZ,CHNZ}
in which centered exterior-product computations away from the 
essential spectrum are observed to require an extremely sparse mesh.

On the other hand, for $\lambda$ inside the essential spectrum,
one or more of the gaps $\nu_\pm$ becomes negative and
$\Re \sigma \mA_\pm$ are no longer of one sign.
As suggested by the simple computations of Remark \ref{egstab},
this might result in a much stricter requirement on $|h_j|$
in order to satisfy condition (ii) (with $\tilde \mu$ now strictly
positive).
Whether this is a real effect or just an artifact of our 
analysis is not clear, but this would be an interesting issue
for further investigation.\footnote{In any case,
this would presumably be confined to shooting methods.}
A second interesting question, for general centered schemes,
would be the relation between the dimension of the kernel of $A_\pm$ 
and the size of the coefficient $C$ in convergence estimate
\eqref{convrate}; we conjecture that they are roughly proportional,
information that could be useful in comparing schemes.

\section{Stability of continuous orthogonalization} \label{orthstab}
We next address stability of the continuous orthogonalization method.
Consider a solution $\bar \Omega$ of the continuous orthogonalization 
system \eqref{Omegaeq}(i) restricted to
the Stiefel manifold $\cS=\{\Omega:\, \Omega^*\Omega=I_k\}$,
such that $\bar \Omega \to \Omega_+$ as $x\to +\infty$.
Evidently, the columns of $\Omega_+$ are an orthonormal basis for
an invariant subspace $\Sigma_+$ of $A_+=\lim_{x\to +\infty}A(x)$.
Of particular interest is the case arising in Evans function computations
that $\Sigma_+$ is the stable or (at certain boundary points)
a neutrally-stable subspace of $A_+$.
Linearizing \eqref{Omegaeq}(i) about $\bar \Omega$,
we obtain the linearized system
$$
\Omega'= (I-\bar \Omega \bar \Omega^*)A\Omega
-\Omega \bar \Omega^* A\bar \Omega
-\bar \Omega \Omega^* A\bar \Omega,
$$
for which the limiting constant-coefficient equation as $x\to +\infty$ is
\be\label{limorth}
\Omega'= \cL \Omega:= (I-\Omega_+ \Omega^*_+)A\Omega
-\Omega \Omega_+^* A\Omega_+
-\Omega_+ \Omega^* A\Omega_+.
\ee
We wish to assess the backward stability of \eqref{limorth}
with respect to {\it general} perturbations,
both along the tangent manifold of the Stiefel manifold $\cS$
and in transverse directions.

\br\label{adjrmk}
 Properly speaking, \eqref{limorth} is linear in
the pair of variables $(\Omega, \Omega^*)$ and not
$\Omega$ alone, since matrix adjoint is not a linear operation.
Along the tangent manifold, however, it is linear as we
shall see in $\Omega$ alone.
\er

\subsection{Asymptotic stability in tangential directions}\label{tan}
The tangent manifold to $\cS$ at $\Omega_+$ consists of
directions $\Omega$ such that $D(\Omega^*\Omega)_{\Omega_+} \Omega=0$,
or
\be\label{skewcond}
\Omega_+^*\Omega + \Omega^*\Omega_+=0,
\ee
that is, for which $\Omega^*_+\Omega$ is skew-symmetric.
As the Stiefel manifold is invariant under \eqref{Omegaeq}(i),
this is evidently invariant under \eqref{limorth}, as direct
computation readily verifies.
It is spanned by the direct sum of
eigenvectors $\Omega_+ K$, $K$ skew, in the kernel
of $\cL$ and eigenvectors 
\be\label{kereig}
\Omega_{jk}:=(I-\Omega_+\Omega_+^*)r_j \sigma_k^*
\ee
in the kernel of $\Omega^*_+$,
where $r_j$ run through the eigenvectors of $A$ transverse to $\Sigma_+$
and $\sigma_k \in \CC^k$ are left eigenvectors of $\alpha$ defined
by $\alpha \Omega_+:= A_+\Omega_+$.
The former correspond to rotations within the same invariant subspace,
the latter to perturbations outside $\Sigma_+$.

Under \eqref{skewcond}, we readily find that \eqref{limorth} simplifies to
\be\label{limorthS}
\Omega'= \cL \Omega:= (I-\Omega_+ \Omega^*_+)
(A\Omega -\Omega \alpha),
\ee
which, for eigenvectors \eqref{kereig} yields
$$
\cL \Omega_{jk}=\Omega_{jk} (a_j-a_k),
$$
where $a_j$ and $a_k$ are the eigenvalue associated with $r_j$
and $\sigma_k$.
Thus, the eigenvalues along the Stiefel manifold are
zero ($k(k+1)/2$-fold) and $a_j-a_k$ ($(n-k)\times k$ in total),
where $a_k$ belong to $\Sigma_+$ and $a_j$ to the compementary
invariant subspace of $A_+$.

In particular, when $\Sigma_+$ is the stable subspace of $A_+$,
the Stiefel eigenvalues of $\cL$ have either zero or positive real part,
and so \eqref{limorth} restricted to the tangent space, as claimed
in the introduction, {\it is (neutrally) stable in backward $x$,}
at least for the limiting system as $x\to +\infty$.

\br\label{linrmk}
Reduced equation \eqref{limorthS} is linear in $\Omega$ alone,
since it does not involve $\Omega^*$.
\er

\subsection{Asymptotic stability in transverse directions}\label{trans}

Off the Stiefel manifold, we face the difficulty pointed out
in Remark \ref{adjrmk} that $\cL$ strictly speaking is a linear
function of $\Omega$ and $\Omega^*$,
so that \eqref{limorth} actually represents a pair of coupled equations,
complicating calculations.
However, we can sidestep much of this difficulty by noting that the
remaining modes not already treated are of form $\Omega_+ \beta$,
where $\beta\in \CC^{k\times k}$ by a brief calculation satisfies
\be\label{betaeq}
\beta'= -(\beta+ \beta^*)\alpha,
\ee
$A_+\Omega_+=:\Omega_+\alpha.$
Introducing $\cR:=\Re \beta:=(1/2)(\beta+\beta^*)$, we thus have
the linear system
\ba\label{trisys}
\cR'&= -\cR \alpha - \alpha^*\cR,\\
\beta'&= -2 \cR \alpha,
\ea
which has block-triangular form 
\be
\bp \cR \\ \beta\ep'
=
\bp 
\cM & 0\\
\cN & 0\\
\ep
\bp \cR \\ \beta\ep,
\ee
where $\cM \cR:= -\cR \alpha - \alpha^*\cR$ and
$\cN \cR:= - 2\cR\alpha$.

This has the $k^2$-dimensional kernel $\cR=0$ already
identified in Section \ref{tan}, and $k^2$ eigenvectors
\be\label{remaining}
\bp \cR_{jk} \\ \cN \cR_{jk}/\mu_{jk}
\ep,
\qquad
\cR_{jk}:= l_jl_k^*,
\ee
with eigenvalues $\mu_{jk}:=-(a_j^* + a_k)$, where
$l_j$ are left eigenvalues of $\alpha$ and $a_j$ the associated
eigenvalues, which are exactly the eigenvalues of $A_+$ restricted
to $\Sigma_+$.

In particular, when $\Sigma_+$ is the stable subspace of $A_+$,
the transverse eigenvalues of $\cL$ have either zero or positive real part,
and so the Stiefel manifold as indicated in the introduction, 
{\it is (neutrally) stable in backward $x$,}
at least for the limiting system as $x\to +\infty$.

\br Alternatively, we may follow the simpler argument of the
introduction to reach the same conclusion without finding explicitly
the eigenmodes of the system.
\er

\subsection{Convergence of the polar coordinate method}\label{polconv}


Consider now the polar coordinate method, consisting of approximation
of \eqref{Omegaeq} on $[-M,0]$ and $[0,M]$ by forward (resp. backward)
Cauchy solvers,
under Dirichlet boundary (i.e., Cauchy) conditions
\be\label{pbc}
\cW^+_J=(\Omega_+,\log \tilde \gamma_+), \quad
\cW^-_J=(\Omega_+,\log \tilde \gamma_+),
\ee
where $\cW^\pm_j$ are the numerical approximations of 
$(\Omega^\pm,\log \tilde \gamma_\pm)(x_j)$.

\bc\label{polarconv}
Under Assumptions \ref{stabassums}, \ref{zcon}, and \ref{exgap},
for fixed $0<\tilde \theta<\theta$,
for $M>0$ sufficiently large and
$\sup_j |h_j| e^{(-\tilde \theta-\tilde \nu) x_j}$ sufficiently small,
where $-\tilde \theta < \tilde \nu \le 0$ is less than $\min \nu_\pm$, 
there exist unique solutions $\cW^\pm$ of 
the discretization of \eqref{Omegaeq} with boundary conditions \eqref{pbc} 
determining an approximate Evans function $\cD^{M,h}(\lambda)$ satisfying
\be\label{polarconvrate}
|\cD^{M,h}-D|\le 
C(e^{-\tilde \theta M } 
+ \sup_j|\tau_j|)
\ee
uniformly on compact subsets of $\Lambda$,
where $D$ is constructed following \eqref{deval}
from the solutions $\WW^\pm$ of \eqref{exeq} 
guaranteed by Lemma \ref{exist}, 
$\theta$ is as in \eqref{expdecay3}, and $\tau_j$ is truncation error.
\ec

\begin{proof}
%
%
Equivalently, we must establish the analog of Theorem \ref{numconv}.
This follows in routine fashion by
decomposing the error equation for the numerical difference scheme into
its linear and nonlinear parts and treating the nonlinear part along
with the conjugation error from transformation into $Z$-coordinates
together as source terms in a fixed-point equation in a combination of
the discrete linear argument of Corollary \ref{exconv} and the
continuous argument of Lemma \ref{gaplemma}, to obtain a contraction
in the weighted $\ell^\infty(\ZZ^+) $ space defined by norm 
$\|\cW\|:=\sup_j |\cW_je^{\tilde \theta x_j}|$, similarly as in
the standard proof of the Stable Manifold Theorem.

We omit the straightforward but tedious details, except to mention one
subtle point that will recur in later applications.
Namely, the righthand side of \eqref{Omegaeq}(i), considered as a function
on the complex-valued matrix $\Omega$, {\it is not $C^1$}.
For, it involves the operation of matrix adjoint, which in turn involves
complex conjugation, a non-analytic function on complex arguments.
Thus, we cannot immediately apply the above-described argument to
\eqref{Omegaeq} as written as a complex-valued ODE, but must instead
first decompose it into real and imaginary parts, or, as we prefer
to do, consider the doubled system
\ba\label{doubled}
\Omega'&= (I-\Omega \tilde \Omega)A\Omega,\\
\tilde \Omega'&= \tilde \Omega A^*(I-\Omega \tilde \Omega)
\ea
in the pair of variables $(\Omega,\tilde \Omega)$, with
$\tilde \Omega:=\Omega^*$.
With this (purely internal) change, the argument goes through
as described to yield the claimed result.\footnote{
Note that applying a standard numerical difference scheme to
the doubled system \eqref{doubled} yields an algorithm
identical to what would be obtained by applying it to
the original equation \eqref{Omegaeq}(i).}
\end{proof}

\section{Boundary-value algorithms}\label{collocation}

Finally, on a more speculative note, we develop further
some ideas of \cite{S,HuZ1} regarding implementation of
boundary-valued based Evans solvers for use in 
extremely large-scale systems, in the light of our new results.

\subsection{Sandstede's method}\label{sand}
We first describe (perhaps an imperfect translation of)
Sandstede's original idea based on established
projective boundary-value methods \cite{Be1}.
This consists, loosely speaking, of numerically solving the original, 
{\it uncentered} Evans system \eqref{gfirstorder} on $[0,M]$
with mixed projective boundary conditions
\be\label{mixed}
\Pi_+ \cW_J^m=0, \quad \Pi_0 \cW_0=\alpha_m,
\quad m=1, \dots, k,
\ee
where $\Pi_+$ is the rank-$(n-k)$ unstable eigenprojection of $A_+$,
$\Pi_0$ is a randomly chosen rank-$k$ projection, and $\alpha_m$,
$m=1, \dots, k$ are $k$ random {\it phase conditions}
determining a basis of $k$ independent solutions $\cW^m$.
Here, as usual, $\cW^m_j$ denotes the numerical approximation
of $W^m(x_j)$.
For generic choices of $\Pi_0$, $\alpha$, this will give a numerically
well-conditioned problem, so the procedure is to randomly select candidate
values, then change these if the method does not converge after appropriate
time.

The solution of the boundary-value scheme for a given parameter value is then
obtained by Newton iteration combined with continuation/path-following. 
Once the method is running, initial guesses for $\Pi_0$, $\alpha$, and
the solution itself may be chosen
strategically based on the solution for nearby parameters, to
improve conditioning/speed of convergence.

The disadvantages of this method are two:
(i) decay at plus infinity
means we don't have control of the asymptotic behavior of
solutions at $+\infty$, making it difficult to impose the desirable
property of analyticity in $\lambda$; indeed, we prescribe solutions by phase
conditions at $x=0$, where we have no direct knowledge of the link
to asymptotic behavior.
(ii) (related) prescription of random phase conditions at the origin 
is logistically complicated, requiring additional error control/programming
beyond just solution of the Evans ODE.

Advantages of the method are the existence of a well-developed 
theory of convergence/error estimation for methods of this form, 
and a hoped-for dimensional advantage of
iterative methods vs. shooting in the treatment of large systems.

%
%

\subsection{A polar coordinate-based method}\label{cont}

Following a suggestion of \cite{HuZ1}, we propose an alternative
boundary-value scheme based on the polar coordinate method, 
using a Newton-based 
iterative boundary-value solver\footnote{For example, MATLAB's BVP5P.}
to approximate the solutions $(\Omega^+, \tilde \Omega^+,\tilde \gamma^+)$
and $(\Omega^-, \tilde \Omega^-,\tilde \gamma^-)$ of the doubled equations
$$
\Omega'= (I-\Omega \tilde \Omega)A_c\Omega,
\qquad
\tilde \Omega'= \tilde \Omega A_c^*(I-\Omega \tilde \Omega),
\qquad
\log \tilde \gamma'= \Trace (\tilde \Omega A_c\tilde \Omega)
-\Trace (\tilde \Omega A_c\tilde \Omega)_\pm
$$
on $[0,\pm M]$ and $[-M,0]$ with Dirichlet boundary conditions
\ba\label{pbc2}
(\Omega, \tilde \Omega, \tilde \gamma)(\pm M)= (\Omega_\pm, \Omega_\pm^*, \tilde \gamma_\pm),
\ea
starting with the exact solution
$(\Omega, \tilde \Omega, \tilde \gamma)
\equiv (\Omega_\pm, \tilde \Omega_\pm, \tilde \gamma_\pm)$
at $c=0$ and continuing via the homotopy
$$
A_c:= A_\pm + c(A-A_\pm), \quad c\in [0,1]
$$
to the desired solution of the full problem $A_c=A$ at $c=1$.

This approach eliminates disadvantages (i)-(ii) of the standard approach and 
appears straightforward to code.
Moreover, Corollary \ref{polarconv} 
gives a rigorous convergence result, indicating at 
least theoretical feasibility. 
What remains to be seen
is whether it is practically useful on the scale of interest, and
how its performance compares with standard uncentered schemes as
described in Section \ref{sand}.
We hope to address these questions in future work.

\br\label{C1}
As noted already in the proof of Corollary \ref{polarconv},
the use of doubled coordinates is necessary in order that the
ODE be $C^1$ as a function of its arguments, since the 
matrix adjoint operation, since not analytic as a complex-valued
function, is not $C^1$.
First-order smoothness is needed to apply Newton iteration, of which
the first step is linearization.
\er

\subsection{A conjugation-based method}\label{conj}
A possible drawback of the polar coordinate in numerically
sensitive situations is its nonlinearity.
An alternative, still more speculative, {linear} solution
would be to use a Newton-based iterative solver to
approximate on $[0,M]$ and $[-M,0]$ solutions $P^+$
and $P^-$ of the conjugation equations
$$
P'=\cA P:= A_\pm P - P A
$$
of \eqref{matrixODE}, using projective boundary conditions
$$
\Pi_\pm (P_\pm-I)(\pm M)=0, \quad \Pi_0 P_\pm=\alpha
$$
starting with initial guess $P^\pm \equiv I$ and using
a similar homotopy 
$$
\cA_c:= \cA_\pm + c(\cA-\cA_\pm), \quad c\in [0,1]
$$
from $\cA$ to its constant-coefficient limits $\cA_\pm$,
defining an Evans approximant simply as
\be\label{approx2}
\cD^{h,M}(\lambda):= \det (P^+R^+, P^-R^-)|_{x=0},
\ee
where $R^\pm$ are matrices whose columns are
bases of the stable (unstable) subspace of
$A_\pm$.\footnote{Described further in Section \ref{init}.}

Again, we have a rigorous convergence result, this time in
the form of Theorem \ref{numconv}, but it is not clear whether
the scheme is practically useful, or if so how its performance
compares to that of the previously mentioned schemes.
Moreover, besides sharing difficulty (ii) of the standard method,
it has the additional difficulty that the dimension of $\Pi_+$
may change with different $\lambda$, necessitating still further
modifications to the boundary conditions. 

\section{Postscript: initialization of eigenbases at infinity}\label{init}
For completeness, we describe, following \cite{BrZ,HSZ,HuZ1,Z2}, a
simple and effective method for computing analytically-chosen initializing
eigenbases at plus and minus spatial infinity.
Combined with the integration methods described in the rest of the paper,
this gives a basic working Evans solver that performs quite 
well in practice.\footnote{
Optimized versions may be found in the STABLAB package developed by
J. Humpherys.}

\subsection{Kato's ODE}
Denote by $\Pi_+$ and $\Pi_-$ the eigenprojections of
$A_+$ onto its stable subspace and $A_-$ onto its unstable
subspace, with $A_\pm$ defined as in \eqref{firstorder}.
Assume as in the introduction 
that the dimensions of the stable and unstable subspaces
are constants $k$ and $n-k$ 
on the desired region of investigation $ \lambda \in \Lambda$
and sum to $n$ (the ``consistent splitting condition'' of \cite{AGJ}).
By standard matrix perturbation theory, 
$\Pi_\pm$ are analytic in $\lambda$ for $\lambda \in \Lambda$ \cite{K}.
Introduce the complex ODE
\be\label{Kato}
R'=\Pi'R, \quad R(\lambda_*)=R_*,
\ee
where $'$ denotes $d/d\lambda$, $\lambda_*\in \Lambda$ is fixed,
$\Pi=\Pi_\pm$, 
and $R=R_\pm$ with $R_+$ and $R_-$ $n\times k$ and $n\times(n-k)$
complex matrices, and $R_*$ is full rank and satisfies 
$\Pi(\lambda_*) R_*=R_*$: that is,
its columns are a basis for the stable (resp. unstable) 
subspace of $A_+$ (resp. $A_-$).

\bl[\cite{K,Z2}]\label{katolem}
There exists a global analytic solution
$R$ of \eqref{Kato} on $\Lambda$
such that (i) $\rank R \equiv \rank R^*$,
(ii) $\Pi R\equiv R$, and (iii) $\Pi R'\equiv 0$.
\el

\begin{proof}
As a linear ODE with analytic coefficients, \eqref{Kato}
possesses an analytic solution in a neighborhood of $\lambda_*$,
that may be extended globally along any curve, whence, by
the principle of analytic continuation, it possesses a global
analytic solution on any simply connected domain containing $\lambda_*$
\cite{K}.
Property (i) follows likewise
by the fact that $R$ satisfies a linear ODE.
Differentiating the identity $\Pi^2=\Pi$ following \cite{K} 
yields $\Pi\Pi'+\Pi'\Pi=\Pi'$, whence, multiplying on the right by $\Pi$,
we find the key property
\begin{equation}\label{Pprop}
\Pi\Pi'\Pi=0.
\end{equation}
From \eqref{Pprop}, we obtain
$$
\begin{aligned}
(\Pi R-R)'&=(\Pi'R+\Pi R'-R')
= \Pi'R+(\Pi-I) \Pi'R
= \Pi \Pi'R,\\
\end{aligned}
$$
which, by $\Pi\Pi'\Pi=0$ and $\Pi^2=\Pi$ gives
$$
\begin{aligned}
(\Pi R-R)'&= -\Pi\Pi'(\Pi R-R), \quad (\Pi R-R)(\lambda_*)=0,
\end{aligned}
$$
from which (ii) follows by uniqueness of solutions
of linear ODE.
Expanding $\Pi R'=\Pi\Pi'R$ and using
$\Pi R=R$ and $\Pi\Pi'\Pi=0$, we obtain $\Pi R'=\Pi\Pi'\Pi R=0$, verifying (iii).
\end{proof}

\br
Property (iii) indicates that the Kato basis is an
optimal choice in the sense that it involves minimal variation
in $R$.
It is also useful as a direct characterization of 
the Kato basis independent of \eqref{Kato}; see \cite{HSZ,BDG}
or Example \ref{sqrteg} below.
\er

\subsection{Numerical implementation}

Choose a set of mesh points $\lambda_j$, $j=0,\dots, J$ 
along a path $\Gamma \subset \Lambda$ and 
denote by $\Pi_j:=\Pi(\lambda_j)$
and $R_j$ the approximation of $R(\lambda_j)$.
Typically, $\lambda_0=\lambda_J$, i.e., $\Gamma$
is a closed contour.

\subsubsection{Computing $\Pi_j$}

Given a matrix $A $, one may efficiently
($\sim 32n^3$ operations; see \cite{GvL,SB}) 
compute by ``ordered'' Schur decomposition\footnote{Supported, for
example, in MATLAB and LAPACK.}, i.e., Schur decomposition
$A=QUQ^{-1}$, $Q$ orthogonal and $U$ upper triangular, for which also
the diagonal entries of $U$ are ordered in increasing real part,
an orthonormal basis 
$$
\check R_u:= (Q_{k+1}, \dots, Q_n),
$$
of its unstable subspace, where $Q_{k+1}, \dots, Q_n$ are the last
$n-k$ columns of $Q$, $n-k$ the dimension of the unstable subspace.
Performing the same procedure for $-A$, $A^*$, and $-A^*$ 
we obtain orthonormal bases $\check R_s$, $\check L_u$, $\check L_s$ 
also for the stable subspace of $A$
and the unstable and stable subspaces of $A^*$, from which
we may compute the stable and unstable eigenprojections 
in straightforward and numerically well-conditioned manner via 
\be\label{Picomp}
\Pi_s:= \check R_s ( \check L_s^* \check R_s )^{-1} \check L_s^*,
\quad
\Pi_u:= \check R_u ( \check L_u^* \check R_u )^{-1} \check L_u^*.
\ee
Applying this to matrices $A_j^\pm:=A_\pm(\lambda_j)$,
we obtain the projectors $\Pi_j^\pm:=\Pi_\pm(\lambda_j)$.
Hereafter, we consider $\Pi_j^\pm$ as known quantities.

\br
It is tempting to instead simply call an eigenvalue--eigenvector 
solver and express $\Pi_s=\sum \frac{r_j l_j^*}{l^*_jr_j}$, 
where $r_j$, $l_j$ are left and right
eigenvalues associated with stable eigenvectors.
However, this becomes ill-conditioned near points where
stable eigenvalues collide, as frequently happens for
cases resulting from other than simple scalar equations.
\er

\subsubsection{First-order integration scheme}
Approximating $\Pi'(\lambda_j)$ to first order by the
finite difference $(\Pi_{j+1}-\Pi_j)/(\lambda_{j+1}-\lambda_j)$
and substituting this into a first-order Euler scheme gives
$$
R_{j+1}=R_j + (\lambda_{j+1}-\lambda_j) 
\frac{\Pi_{j+1}-\Pi_j}{\lambda_{j+1}-\lambda_j}
R_j,
$$
or
$R_{j+1}=R_j + \Pi_{j+1}R_j - \Pi_jR_j$, 
yielding by the property $\Pi_jR_j=R_j$ (preserved exactly
by the scheme) the simple greedy algorithm
\be\label{greedy}
R_{j+1}=\Pi_{j+1}R_j.
\ee
It is a remarkable fact \cite{Z2} (a consequence of Lemma \ref{katolem})
that, up to numerical error, evolution of \eqref{greedy}
about a closed loop
$\lambda_0=\lambda_J$ yields the original value $R_J=R_0$.

\subsubsection{Second-order scheme}

To obtain a second-order discretization of \eqref{Kato},
we approximate
$ R_{j+1}-R_j \approx \Delta \lambda_j \Pi'_{j+1/2}R_{j+1/2}$,
good to second order, where $\Delta \lambda_j:=\lambda_{j+1}-\lambda_j$.  
Noting that $R_{j+1/2}\approx \Pi_{j+1/2}R_j$ to second order, by
\eqref{greedy}, and approximating
$ \Pi_{j+1/2}\approx \frac{1}{2}(\Pi_{j+1}+\Pi_j)$, also good to second order,
and  
$\Pi'_{j+1/2}\approx (\Pi_{j+1}-\Pi_j)/\Delta \lambda_j$,
we obtain, combining and rearranging,
$$
R_{j+1}=R_j + \frac{1}{2}(\Pi_{j+1}-\Pi_j)(\Pi_{j+1}+ \Pi_j)R_j.
$$
Stabilizing by following with a projection $\Pi_{j+1}$, we obtain
after some rearrangement the reduced second-order explicit scheme
\begin{equation}\label{greedy2}
R_{j+1}= \Pi_{j+1}[ I + \frac{1}{2}\Pi_j (I-\Pi_{j+1})] R_j.
\end{equation}
This is the version we recommend for serious computations.
For individual numerical experiments the simpler greedy 
algorithm \eqref{greedy} will often suffice (see discussion,
\cite{Z2}).

\br
Arbitrarily higher-order schemes may be obtained 
by Richardson extrapolation starting from scheme 
\eqref{greedy} or \eqref{greedy2}; see \cite{Z2}.
In practice, this does not seem useful.
\er

\subsection{Initialization of Evans function ODE}\label{evinit}
Finally, we describe the conversion of
analytic bases $R_\pm(\lambda)$ into initial data for the centered
exterior product or polar coordinate method.

\subsubsection{Centered exterior product scheme}\label{exinit}
Denote 
$R^+=(R_1^+,\dots, R_k^+)$
and $ R^-=(R_1^-,\dots, R_{n-k}^-)$.
Then, the initializing wedge products
$\cR_{S_\pm}$ of Section \ref{extstab} are given simply by
\be\label{exinitform}
\cR_{S_+}:=R_1^+\wedge \cdots\wedge R_k^+
\: \hbox{\rm and } \:
\cR_{S_-}:=R_1^-\wedge \cdots\wedge R_{n-k}^-.
\ee

\subsubsection{Polar coordinate scheme}\label{polarinit}

For each $\lambda\in \Lambda$,
we may efficiently compute matrices $\Omega_\pm(\lambda)$ 
whose columns form orthonormal bases for $S_\pm$, 
by the same ordered Schur decomposition used in the computation
of $\Pi_\pm$.
This need not even be continuous with respect to $\lambda$.  
Equating
\[
\Omega_+ \tilde\alpha_+(\lambda)= R_+(\lambda),
\quad
\Omega_- \tilde\alpha_-(\lambda)= R_-(\lambda),
\]
for some $\tilde \alpha_\pm$, we obtain
\[
\tilde\alpha_+(\lambda)= \Omega^*_+ R_+(\lambda),
\quad
\tilde\alpha_-(\lambda)= \Omega^*_- R_-(\lambda),
\]
and therefore the exterior product of the columns of $R_\pm$ 
is equal to the exterior product of the columns of $\Omega_\pm$ times 
\begin{equation}
\label{gamma-}
\tilde\gamma_\pm(\lambda):= \det(\Omega^*R)_\pm(\lambda)).
\end{equation}
Thus, we may initialize the polar coordinate ODE \eqref{Omegaeq} with
\be\label{polariniteq}
\Omega=\Omega_\pm, \quad \tilde \gamma= \det(\Omega^*R)_\pm.
\ee

\br
The wedge products represented by polar coordinates 
$(\tilde\gamma, \Omega)_\pm( \lambda)$, with 
$\tilde \gamma_\pm$ defined as in \eqref{gamma-}, are the same products
$\cR_{S_\pm}$ defined in \eqref{exinitform}. 
In particular, they are analytic with respect to $\lambda$,
though coordinates $\tilde \gamma$ and $\Omega$ in general are not.
\er

\subsection{Error control}\label{katoerror}

As the integration of Kato's ODE is carried out on a bounded closed
curve, standard error estimates apply and convergence is essentially
automatic, and we shall not discuss it. 
In applications, we are often interested in determining the winding
number of $D$ about such a curve.
For this purpose, following \cite{Br,BrZ}, 
we introduce a simple a posteriori ``Rouch\'e'' check
limiting the step-size in $\lambda$ by the requirement that
the relative change in $D(\lambda)$ be less than a conservative $0.1$.
(By Rouch\'e's Theorem, relative error less than one is sufficient
to obtain the correct winding number.)
In contrast to integration in $x$ of the Evans system, integration
in $\lambda$ of the Kato system is a one-time cost, 
so not a rate-determining factor in the performance of the overall code.
However, the computation time {\it is} sensitive (proportional) to the
number of mesh points in $\lambda$, so this should be held down as much
as possible.

\subsection{Finer points: two exceptional cases}\label{exceptional}

We conclude by pointing out two commonly occurring cases
that can give trouble if not expected, and describe some 
practical resolutions.

\subsubsection{Behavior near the origin}\label{origin}

For the linearized operators $L$ arising in the study of stability
of traveling-waves of certain systems such
as viscous conservation laws or Cahn--Hilliard and 
nonlinear Schr\"odinger equations, the point $\lambda=0$ is
embedded in the essential spectrum of $L$.
In computing a winding number around some bounded portion of the
set $\{\Re \lambda \ge 0\}$ of possible unstable eigenvalues,
we must pass through or near this value, at which the spectral gap
(see discussion above Assumption \ref{exgap}) between stable and
unstable subspaces of $A_\pm$ goes to zero.
In this case, the eigenprojections $\Pi_\pm$ lose their characterization
as stable (resp. unstable) eigenprojections of $A_\pm$, so must be
computed in a different way than the ordered Schur decomposition
described above.
Worse, they may lose analyticity, possessing a branch singularity,
at $\lambda=0$.

To avoid the former problem, we typically just compute near but not 
at $\lambda=0$.
However, this leads to occasional uncertainty/bad results near the origin
and should probably be improved in the analytic case 
by instead computing $\Pi_\pm$
at points within or on the essential spectrum boundary of $L$
by analytic extrapolation from values at points outside.
This is an important practical area for further algorithm 
development.\footnote{This would also allow computations within
the essential spectrum, up to now not systematically carried out
(though see \cite{Br} for some preliminary results in this direction).}
The latter problem, concerning behavior near a branch singularity,
is discussed in the next subsection.


\subsubsection{Behavior near a branch singularity}


For certain problems, especially those involving additional parameters,
e.g. multi-d \cite{HLyZ2} or families
of one-dimensional waves that pass through characteristic values
\cite{BHZ}, there may appear for certain parameters
branch points in the eigenvalues of
$A_\pm$ as a function of $\lambda$, 
at which $\Pi_\pm$ therefore blow up \cite{K}.
This requires some adustment in order to restore good behavior.

\begin{example}\label{sqrteg}
A model for this situation is the eigenvalue equation for a
scalar convected heat equation
$\lambda u + \eta u'=u''$ with
convection coefficient $\eta$ passing through zero.
The coefficient matrix for the associated first-order system 
is 
\be\label{egA}
A:=\bp 0 & 1 \\ \lambda & \eta\ep.
\ee
Then, the stable eigenvector of $A$ determined by Kato's ODE \eqref{Kato} is
\be\label{charkato}
R(\eta, \lambda):=
\frac{ (\eta^2/4+ 1)^{1/4}} { (\eta^2/4+ \lambda)^{1/4}}
\Big(1, -\eta/2 - \sqrt{\eta^2/4 + \lambda}\, \Big)^T,
\ee
which, apart from the divergent factor 
$\frac{ (\eta^2/4+ 1)^{1/4}} { (\eta^2/4+ \lambda)^{1/4}}$,
is a smooth function of $ \sqrt{\eta^2/4 + \lambda}$.
\end{example}

\begin{proof}
By straightforward computation, 
$\mu_\pm(\lambda):= \mp(\eta/2 +\sqrt{\eta^2/4 + \lambda}$
and $\mathcal{V}_\pm:=(1, \mu_\pm(\lambda))^T$
are eigenvalues and eigenvectors of 
the matrix $A$ of \eqref{egA} in Example \ref{sqrteg}. 
The associated Kato eigenvectors $V^\pm$
are determined uniquely, up to a constant factor
independent of $\lambda$, by the property that there
exist corresponding left eigenvectors $\tilde V^\pm$ such that
\begin{equation}
\label{katoprop2}
(\tilde V\cdot V)^\pm \equiv  {\rm constant}, \quad
(\tilde V \cdot \dot V)^\pm \equiv 0,
\end{equation}
where ``$\, \, \dot{    }\, \,$'' denotes $d/d\lambda$;
see Lemma \ref{katolem}(iii).

Computing dual eigenvectors $\tilde{\mathcal{V}}^\pm =
(\lambda+\mu^2)^{-1}(\lambda, \mu_\pm)$ satisfying
$(\tilde{\mathcal{V}}\cdot \mathcal{V})^\pm \equiv 1$, 
and setting $V^\pm=c_\pm\mathcal{V}^\pm$, $\tilde V^\pm=\mathcal{V}^\pm/c_\pm$,
we find after
a brief calculation that \eqref{katoprop2} is equivalent to the complex ODE
\begin{equation}\label{kato_ode}
\begin{aligned}
\dot c_\pm &= -\Big( \frac{ 
\tilde V \cdot \dot V
}{ 
\tilde V \cdot V
}\Big)^\pm c_\pm 
=
-\Big( \frac{ 
\dot \mu
}{ 
2\mu - \eta
}\Big)_\pm c_\pm,
\end{aligned}
\end{equation}
which may be solved by exponentiation, yielding the general solution
$ c_\pm (\lambda)= C(\eta^2/4+\lambda)^{-1/4}.  $
Initializing without loss of generality at $c_\pm(1)=1$,
we obtain \eqref{charkato}.
\end{proof}

\br
It is straightforward to generalize by the same method
the computation of Example \ref{sqrteg}
to branch singulariteis of general order $s$.
\er

The computation of Example \eqref{sqrteg} indicates that the
Kato basis blows up at $\lambda=0$ as
$\big( {\eta}^2+4\lambda\big)^{-1/4}$ as
$\eta$ crosses the characteristic value $\eta=0$,
hence does not give a choice
that is continuous across the entire range of parameters.
However, the same example shows that there is a different choice
$(1, -\eta/2 - \sqrt{\eta^2/4 + \lambda}\, )^T$
that {\it is} continuous, possessing only a square-root singularity.
We can effectively exchange one for another, by rescaling the
Kato basis by factor $\big( {\eta}^2+4\lambda\big)^{1/4}$.
See \cite{BHZ} for examples/further details.

This issue is mainly important in problems with parameters, or
possessing branch singularities, but
can also arise for a problem without parameter or singularity that happens
to lie near a related problem with branch singularities.
In such a case the ``invisible'' branch singularity could serve
as an organizing center directing the Kato flow without the user
being aware of it.
Thus, it is important to be alert to this possibility.

\end{document}